\documentclass[a4paper,12pt]{amsproc}

\usepackage{amssymb}
\usepackage{graphicx}
\usepackage{float}
\usepackage{amsmath}
\usepackage{array}
\usepackage{multirow}
\usepackage{mathrsfs}
\usepackage{textcomp}
\usepackage[bottom]{footmisc}
\usepackage{xcolor}
\usepackage{cite}

\newcolumntype{M}[1]{>{\raggedright}m{#1}}

\DeclareMathAlphabet{\mathpzc}{OT1}{pzc}{m}{it}

\newtheorem{theorem}{Theorem}[section]
\newtheorem{lemma}[theorem]{Lemma}
\newtheorem{proposition}[theorem]{Proposition}
\newtheorem{corollary}[theorem]{Corollary}

\newtheorem{claim}[theorem]{Claim}

\theoremstyle{definition}
\newtheorem{definition}[theorem]{Definition}
\newtheorem{example}[theorem]{Example}

\theoremstyle{remark}
\newtheorem{remark}[theorem]{Remark}

\numberwithin{equation}{section}

\allowdisplaybreaks[4]

\begin{document}

\title{Newton weighted-L\^e-Yomdin polynomials and $\mu$-Zariski pairs of surface singularities}

\author{Christophe \textsc{Eyral}, Masaharu \textsc{Ishikawa}, Mutsuo \textsc{Oka}}

\address{C. Eyral, Institute of Mathematics, Polish Academy of Sciences, ul. \'Sniadeckich~8, 00-656 Warsaw, Poland}  
\email{cheyral@impan.pl} 
\address{M. Ishikawa, Department of Mathematics, Hiyoshi Campus, Keio University,  4-1-1 Hiyoshi, Kohoku-ku, Yokohama, Kanagawa 223-8521, Japan}
\email{ishikawa@keio.jp}
\address{M. Oka, Emeritus Professor, Tokyo Institute of Technology, 2-12-1 Ohokayama, Meguro-ku, Tokyo 152-8551, Japan}   
\email{okamutsuo@gmail.com}

\thanks{}


\begin{abstract}
We investigate surface singularities defined by weighted-L\^e-Yomdin polynomials, with a particular focus on a specific subclass that we refer to as \emph{Newton weighted-L\^e-Yomdin polynomials}. In particular, using polynomials in this subclass, we develop a method to construct new $\mu$-Zariski pairs of surface singularities.
\end{abstract}

\null\vskip -1cm

\maketitle

\markboth{C. Eyral, M. Ishikawa, M. Oka}{Newton weighted-L\^e-Yomdin polynomials and $\mu$-Zariski pairs of surface singularities}

\section{Introduction}

Let $g_0$ and $g_1$ be two polynomials in three complex variables $x,y,z$. Assume that $g_0$ and $g_1$ vanish at the origin $\mathbf{0}\in\mathbb{C}^3$ and that the corresponding germs of surfaces, $V(g_0)$ and $V(g_1)$, have an isolated singularity at $\mathbf{0}$. In \cite{Yau}, Yau conjectured that if $g_0$ and $g_1$ have the same monodromy zeta-function---equivalently, the same characteristic polynomial---and if $V(g_0)$ and $V(g_1)$ are homeomorphic, then the pairs of germs $(\mathbb{C}^3,V(g_0))$ and $(\mathbb{C}^3,V(g_1))$ are homeomorphic too. In \cite{A2,A1}, Artal Bartolo disproved this conjecture, showing that, in general, the monodromy zeta-function and the abstract topology do not determine the embedded topology. Based on his counter-example, we define a pair of polynomials sharing the same monodromy zeta-function as a \emph{Zariski pair of surface singularities in $\mathbb{C}^3$} if the corresponding germs of surfaces have identical abstract topology while exhibiting distinct embedded topologies.

More specifically, in the aforementioned references, Artal Bartolo presents a general method for constructing such Zariski pairs of surfaces, which is detailed as follows. Suppose that $g_0$ and $g_1$ define \emph{superisolated} surface singularities, that is, $g_0=f_0+x^{d+1}$ and $g_1=f_1+x^{d+1}$, where $f_0$ and $f_1$ are reduced homogeneous polynomials of degree $d$, and where the projective tangent cones $C_0$ and $C_1$ of $V(g_0)$ and $V(g_1)$---which are defined by the zero sets in $\mathbb{P}^2$ of $f_0$ and $f_1$, respectively---have their singular points away from the projective line $x=0$ (see \cite{Luengo}). Assune further that $C_0$ and $C_1$ form a \emph{Zariski pair of projective curves in $\mathbb{P}^2$}. By such a pair, we mean that there are regular neighbourhoods $N_0$ and $N_1$ of $C_0$ and $C_1$, respectively, such that $(N_0,C_0)$ and $(N_1,C_1)$ are homeomorphic\footnote{Equivalently, $C_0$ and $C_1$ have the same ``combinatoric'' in the sense of \cite{Artal,ACT}. In particular, this implies that $C_0$ and $C_1$ have the same singularities (up to topological equivalence).}\label{footnote1} while $(\mathbb{P}^2,C_0)$ and $(\mathbb{P}^2,C_1)$ are not homeomorphic (see \cite{Z,Artal})---Zariski pairs of projective curves in $\mathbb{P}^2$ should not be confused with Zariski pairs of surface singularities in $\mathbb{C}^3$. Then, by \cite{A,GZ-L-MH,O1,O3,Siersma,Siersma2,Stevens}, the monodromy zeta-functions of $g_0$ and $g_1$ coincide, and by \cite{Luengo}, $V(g_0)$ and $V(g_1)$ have the same abstract topology. In \cite{A2, A1}, Artal Bartolo showed that if the \emph{Alexander polynomials} of the curves $C_0$ and $C_1$ are different, then the Jordan forms of the monodromies associated with $V(g_0)$ and $V(g_1)$, respectively, are distinct, and consequently the pairs $(\mathbb{C}^3,V(g_0))$ and $(\mathbb{C}^3,V(g_1))$ are not homeomorphic. 

When the Alexander polynomials of the curves $C_0$ and $C_1$ coincide, knowledge in this area remains rather limited. An initial step in this direction was made by the first and third named authors in \cite{EO}, who constructed the first \emph{$\mu^*$-Zariski pair of surface singularities} whose projective tangent cones may have the same Alexander polynomial.\footnote{Examples of Zariski pairs of projective curves sharing the same Alexander polynomial can be found, for instance, in \cite{AC,O-1,O-2,Deg,EO4,EO5}.}  Here, by $\mu^*$-Zariski pair, we refer to a pair of surface singularities, $g_0$ and $g_1$, which share the same abstract topology, the same monodromy zeta-function, and the same Teissier $\mu^*$-invariant, yet lie in distinct path-connected components of the $\mu^*$-constant stratum. Of course, being such a pair does not imply for $V(g_0)$ and $V(g_1)$ to have distinct embedded topologies. However, by \cite{Teissier2}, it is a \emph{necessary} condition for achieving that outcome and represents a significant step toward understanding the problem. 

The construction presented in \cite{EO} is also general in nature and unfolds as follows. Assume that $g_0$ and $g_1$ define \emph{L\^e-Yomdin} surface singularities, that is, $g_0=f_0+x^{d+m}$ and $g_1=f_1+x^{d+m}$, where $f_0$ and $f_1$ are reduced homogeneous polynomials of degree $d$ such that the singularities of the projective tangent cones $C_0$ and $C_1$ of $V(g_0)$ and $V(g_1)$ are not located on the projective line $x=0$ and where $m>0$ is a positive integer (see \cite{I,LeDT}). In fact, by a linear change of coordinates, we can always assume without loss of generality that the singularities of $C_0$ and $C_1$ are located in $\{xyz\not=0\}$, and that $f_0$ and $f_1$ are convenient and Newton non-degenerate on any proper face of their (common) Newton boundary, in the sense of Kouchnirenko \cite{K}. As above, we suppose that $C_0$ and $C_1$  form a Zariski pair of projective curves in~$\mathbb{P}^2$. 
Now, if we assume further that the singularities of $C_0$ and $C_1$ are Newton non-degenerate in suitable local coordinates,\footnote{For example, this is always the case if the singularities are ``simple'' in the sense of Arnol'd \cite{Arnold}.} then we easily check, using \cite{O3}, that the Teissier $\mu^*$-invariant and the monodromy zeta-functions of $g_0$ and $g_1$ coincide, and by \cite{O2}, that $V(g_0)$ and $V(g_1)$ also have the same abstract topology. In \cite{EO}, we showed that under this additional Newton non-degeneracy assumption, $g_0$ and $g_1$ necessarily also lie in distinct path-connected components of the $\mu^*$-constant stratum, and consequently form a $\mu^*$-Zariski pair of surface singularities. 

We expect that the pairs constructed using this method will, in fact, produce Zariski pairs of surface singularities, as defined above. Again, we highlight that in this construction, the curves $C_0$ and $C_1$ may have the same Alexander polynomial. 

In this paper, we develop a new technique for constructing \emph{$\mu$-Zariski pairs of surface singularities} (see below for the definition)\footnote{The first examples of such pairs were obtained in \cite{O2}, albeit under a slightly different definition than the one adopted here.}, beginning with \emph{Zariski pairs of curves in a weighted projective plane} $\mathbb{P}^2(P)$.\footnote{Roughly speaking, a \emph{Zariski pair of curves in  $\mathbb{P}^2(P)$} consists of two curves with the same combinatoric, while possessing non-homeomorphic embeddings in at least one coordinate subspace of $\mathbb{P}^2(P)$ (for a precise definition, see Definition \ref{def-zpwpc}). Such pairs have also been studied in~\cite{ACM}, albeit under a slightly different definition than the one employed here.}  Here, $P={}^t(p_1,p_2,p_3)$ is a primitive\footnote{That is, $\gcd(p_1,p_2,p_3)=1$.} weight vector, and $\mathbb{P}^2(P)$ denotes the corresponding projective plane. More precisely, we work within the following framework. Assume that $g_0$ and $g_1$ are of the form $g_0=f_0+h_0$ and $g_1=f_1+h_1$, where $f_0$ and $f_1$ are the weighted homogeneous initial polynomials of $g_0$ and $g_1$ respectively (with respect to some fixed weight vector $P$). We assume that $f_0$ and $f_1$ have the same weighted degree $d$, the same Newton boundary $\Delta$, and that they are Newton non-degenerate on every proper face of $\Delta$. We also suppose that the Newton polyhedra of $h_0$ and $h_1$ are included in that of $f_0$ and $f_1$, respectively, and that the smallest weighted degrees of their monomials are the same, say equal to $d+m$ for some $m>0$. Finally, we require $g_0$ and $g_1$ to define \emph{weighted-L\^e-Yomdin} singularities at~$\mathbf{0}$, that is, for $j=0,1$ the singular locus of the weighted projective curve $f_j=0$ in $\mathbb{P}^2(P)$ does not meet the curve defined by the weighted homogeneous initial form ($P$-principal part) of $h_j$ (see \cite{ABLMH,ACM}).
In fact, we impose a slightly weaker requirement: specifically, we only ask that the previous condition holds on $\mathbb{P}^2(P) \setminus \{xyz = 0\}$ (see section \ref{nwlyp} below for details).
Hereafter, we will refer to such polynomials $g_0$ and $g_1$ as \emph{weak-Newton weighted-L\^e-Yomdin polynomials} and we will denote their set by $\mathcal{W}_{P,d,m}''(\Delta)$. (Here, ``weak-Newton'' refers to the Newton non-degeneracy of $f_0$ and $f_1$ with respect to the proper faces of $\Delta$.) In the special case where the singularities of the weighted projective curves $f_0=0$ and $f_1=0$ are also Newton non-degenerate (again, see section \ref{nwlyp} for the precise meaning of this), then we will say that $g_0$ and $g_1$ are \emph{Newton weighted-L\^e-Yomdin polynomials}. 

Within this setting, we show that if $f_0$ and $f_1$ form a Zariski pair of curves in $\mathbb{P}^2(P)$ with Newton non-degenerate singularities, then, up to a technical condition on the weights in one special case, the Newton weighted-L\^e-Yomdin polynomials $g_0=f_0+h_0$ and $g_1=f_1+h_1$  form a \emph{$\mu$-Zariski pair of surface singularities in $\mathbb{C}^3$} (see Theorem \ref{mt3}). Here, by $\mu$-Zariski pair, we mean that $g_0$ and $g_1$ have the same abstract topology, the same monodromy zeta-function and the same Teissier $\mu^*$-invariant, but belong to distinct path-connected components of the $\mu$-constant stratum of $\mathcal{W}_{P,d,m}''(\Delta)$. In particular, $g_0$ and $g_1$ also belong to distinct path-connected components of the $\mu^*$-constant stratum of $\mathcal{W}_{P,d,m}''(\Delta)$.

The proof of this theorem is based on three other results presented in this paper and having also  their own interests. The first one concerns the Milnor number of a generic plane section of a weak-Newton weighted-L\^e-Yomdin polynomial $g=f+h\in \mathcal{W}_{P,d,m}''(\Delta)$ and of its weighted homogeneous initial form $f$ (see Theorems \ref{remthm1} and \ref{mt1}). It is used to compare the $\mu^*$-invariant of $g_0=f_0+h_0$ and $g_1=f_1+h_1$ in  Theorem \ref{mt3}. The second result deals with \emph{$\mu$-constant} piecewise-complex analytic one-parameter families $\{g_s\}_{0\leq s\leq 1}$ of weak-Newton weighted-L\^e-Yomdin polynomials $g_s=f_s+h_s\in \mathcal{W}_{P,d,m}''(\Delta)$. It asserts that for any such a family, the singularities of the weighted projective curves defined by the corresponding family $\{f_s\}_{0\leq s\leq 1}$ of  weighted homogeneous initial forms remain stable as the parameter $s$ varies (see Theorem \ref{mt2}). Finally, the third result says that for such a family $\{f_s\}_{0\leq s\leq 1}$, the weighted projective curves $f_0=0$ and $f_1=0$ obtained for $s=0$ and $s=1$, respectively, cannot form a Zariski pair of curves in $\mathbb{P}^2(P)$ (see Proposition \ref{prop-samecc}).
These two last properties are used to show that the polynomials $g_0$ and $g_1$ of Theorem \ref{mt3} belong to distinct path-connected components of the $\mu$-constant stratum of $\mathcal{W}_{P,d,m}''(\Delta)$.

The proof of Theorem \ref{mt3} also relies on the fact that the usual (continuous) path-connected components of $\mathcal{W}_{P,d,m}''(\Delta)$ coincide with its \emph{piecewise complex-analytic} path-connected components (see Remark~\ref{rem-cpca}). This follows from \cite[Proposition 5.2 and Theorem 5.4]{EO2}, together with the fact that $\mathcal{W}_{P,d,m}''(\Delta)$ is a \emph{constructible} set. While the constructibility of this set is straightforward, that of the subset consisting of those elements $g=f+h$ in $\mathcal{W}_{P,d,m}''(\Delta)$ for which the weighted projective curve defined by $f=0$ admits a fixed prescribed configuration of singularities $\Xi$, is more subtle. This latter issue is addressed in Proposition \ref{prop-constructible}, which constitutes one of the key ingredients of the proof of Proposition~\ref{prop-samecc} mentioned above.

Ultimately, to methodically construct examples of Zariski pairs of curves in weighted projective spaces---which is the initial step of our construction of $\mu$-Zariski pairs of surface singularities---we present a method originating from \cite{O5}. This method  associates to each classical \emph{strong}\footnote{By a \emph{strong} Zariski pair of curves in $\mathbb{P}^2$, we mean a Zariski pair in which the embedded topologies of the curves are distinguished by the fundamental groups of their complements.} Zariski pair of curves  in the standard projective plane, a corresponding Zariski pair of curves in a special weighted projective plane (see Theorem \ref{mt4bis} and Corollary \ref{mt4}).

\section{Setting and preliminary observations}

In this section, we set up the geometric framework in which we will operate, and make some observations that will be useful later on. The notations introduced here are used consistently throughout the paper.

\subsection{Weighted projective plane and affine coordinate charts}\label{sect-wpswpc}

Let us fix, once and for all, coordinates $(x,y,z)$ on $\mathbb{C}^3$, as well as a vector $P={}^t(p_1,p_2,p_3)\in\mathbb{N}^3$ of positive weights on the variables $x,y,z$ such that $p_1\geq p_2\geq p_3>0$ and $\gcd(p_1,p_2,p_3)=1$. The space of orbits under the $\mathbb{C}^*$-action on $\mathbb{C}^3\setminus\{\mathbf{0}\}$ defined by
\begin{equation*}\label{csa}
t\cdot(x,y,z):=(t^{p_1}x,t^{p_2}y,t^{p_3}z)
\end{equation*}
is called the \emph{weighted projective plane of type $P$} and is denoted by $\mathbb{P}^2(P)$. It is well known that $\mathbb{P}^2(P)$ is a normal variety having only quotient singularities, all located in $\{xyz=0\}$.

Observe that if $p_3=1$, then the open set 
\begin{equation*}
U_z:=\mathbb{P}^2(P)\setminus\{z=0\}
\end{equation*}
is non-singular, isomorphic to $\mathbb{C}^2$ with coordinates $u:=x/z^{p_1}$ and $v:=y/z^{p_2}$ by the mapping $[x\!:\!y\!:\!z]_P\in U_z\mapsto (u,v)\in\mathbb{C}^2$ with inverse $(u,v)\in\mathbb{C}^2\mapsto [u\!:\!v\!:\!1]_P\in U_z$. Here, $[\dots]_P$ denotes the equivalence class in $\mathbb{P}^2(P)$. In this case, we will say that $(U_z,(u,v))$ is an \emph{affine coordinate chart} of $\mathbb{P}^2(P)$. Similarly, if $p_2=1$ or $p_1=1$, then the open sets 
\begin{equation*}
U_y:=\mathbb{P}^2(P)\setminus\{y=0\}
\quad\mbox{or}\quad
U_x:=\mathbb{P}^2(P)\setminus\{x=0\}
\end{equation*}
are also isomorphic to $\mathbb{C}^2$, with coordinates $u':=x/y^{p_1}$, $w':=z/y^{p_3}$ for $U_y$, and $v'':=y/x^{p_2}$, $w'':=z/x^{p_3}$ for $U_x$, and $(U_y,(u',w'))$ and $(U_x,(v'',w''))$ are also called affine coordinate charts. Note that if $p_3>1$, then the assumption $p_1\geq p_2\geq p_3$ implies that $\mathbb{P}^2(P)$ has no affine coordinate chart.

Finally, observe that if $f(x,y,z)$ is a complex, reduced, weighted homogeneous polynomial of degree $d$ with respect to $P$ (in short, of $P$-degree $d$) with $p_3=1$, and if $V_P(f)$ is the weighted projective curve defined by $f$ in $\mathbb{P}^2(P)$, then, in the affine coordinate chart $(U_z,(u,v))$,  the curve $V_P(f)$ is defined by $f(u,v,1)=0$.

\subsection{Weighted homogenization}\label{subsect-wh}

In this section, we assume that $p_3=1$. Let 
\begin{equation*}
F(x,y)=\sum_{(a,b)}\nu_{a,b}\, x^ay^b
\end{equation*}
be a (not necessarily weighted homogeneous) polynomial of degree $d$ with respect to the weight vector $P':={}^t(p_1,p_2)$, and let 
\begin{equation*}
F_d(x,y):=\sum_{p_1a+p_2b=d}\nu_{a,b}\, x^ay^b.
\end{equation*}
Here, of course, $d$ is defined by $d:=\sup\{p_1a+p_2b\, ;\, \nu_{a,b}\not=0\}$. The \emph{weighted homogenization of $F$ with respect to $P={}^t(p_1,p_2,1)$} is the polynomial $f$ defined by 
\begin{equation*}
f(x,y,z):=z^d\cdot F\Big(\frac{x}{z^{p_1}},\frac{y}{z^{p_2}}\Big).
\end{equation*}
Clearly, $f$ is a weighted homogeneous polynomial of $P$-degree $d$, and
\begin{equation*}
f(0,y,z)=z^d F(0,y/z^{p_2}),\quad f(x,0,z)=z^d F(x/z^{p_1},0), \quad f(x,y,0)=F_d(x,y).
\end{equation*}

Now, consider a reduced homogeneous polynomial $h(x,y,z)$ of degree $d$. By a linear change of coordinates, we may assume that $h(0,y,z)$, $h(x,0,z)$ and $h(x,y,0)$ are convenient and Newton non-degenerate. Let 
\begin{equation*}
H(x,y):=h(x,y,1)
\quad\mbox{and}\quad
F(x,y):=H(x,y^{p_1}). 
\end{equation*}
Clearly, the degree of $F$ with respect to ${}^t(p_1,1)$ is $p_1d$. Consider the weighted homogenization $f$ of $F$ with respect to the weight vector ${}^t(p_1,1,1)$, that is,
\begin{equation*}
f(x,y,z):=z^{p_1d}\cdot F\Big(\frac{x}{z^{p_1}},\frac{y}{z}\Big),
\end{equation*}
which is a weighted homogeneous polynomial of degree $p_1d$ with respect to ${}^t(p_1,1,1)$. Clearly, we have
\begin{align*}
& f(0,y,z)=z^{p_1d} F(0,y/z)=z^{p_1d} H(0,(y/z)^{p_1}),\\
& f(x,0,z)=z^{p_1d} F(x/z^{p_1},0)=z^{p_1d} H(x/z^{p_1},0),\\
& f(x,y,0)=F_{p_1d}(x,y)=H_d(x,y^{p_1}).
\end{align*}
It follows that these face-functions are Newton non-degenerate.
Indeed, since $h(0,y,z)$, for instance, is convenient and Newton non-degenerate, it can be written as $h(0,y,z)=\prod_{i=1}^d(y-\alpha_i z)$, and therefore $f(0,y,z)=\prod_{i=1}^d(y^{p_1}-\alpha_i z^{p_1})$, where $\alpha_1,\ldots,\alpha_d$ are mutually distinct.

\section{Milnor number and generic plane section}\label{sect-mngps}

Let $\mathcal{W}_{P,d}$ be the space of complex, reduced, weighted homogeneous polynomials (in three variables) of degree $d$ with respect to $P$, and for any $f(x,y,z)\in \mathcal{W}_{P,d}$, let $V(f)$ denote the \emph{affine surface} defined by $f$ in $\mathbb{C}^3$. As above, we denote by $V_P(f)$ the weighted projective curve defined by $f$ in $\mathbb{P}^2(P)$, and we put
\begin{equation}\label{def-p2s}
\mathbb{P}^{2*}(P):=\mathbb{P}^2(P)\setminus\{xyz=0\}.
\end{equation}
Note that since the singular locus $\mbox{Sing}(V(f))$ of $V(f)$ is stable under the $\mathbb{C}^*$-action with respect to $P$, the intersection $V_P(f)\cap \mathbb{P}^{2*}(P)$ has only (a finite number of) \emph{isolated} singularities if $\mbox{Sing}(V(f))\subseteq\mathbb{C}^{*3}\cup\{\mathbf{0}\}$.

\subsection{The sets $\mathcal{W}_{P,d}(\Delta)$ and $\mathcal{W}'_{P,d}(\Delta)$}\label{sect-mngps-31}
Now, take a Newton polyhedron $\Gamma_+\subseteq\mathbb{R}^3_+$ whose Newton boundary has a unique $2$-dimensional (compact) face $\Delta$, and consider the subspaces 
\begin{equation*}
\mathcal{W}_{P,d}(\Delta)
\quad\mbox{and}\quad
\mathcal{W}'_{P,d}(\Delta)
\end{equation*}
of $\mathcal{W}_{P,d}$ defined as~follows:
\begin{enumerate}
\item[$\bullet$]
$f \in \mathcal{W}_{P,d}(\Delta)$ if and only if $f$ is Newton non-degenerate, possesses an isolated critical point at $\mathbf{0}$ and satisfies $\Gamma(f) = \Delta$, where $\Gamma(f)$ denotes the Newton boundary of $f$;
\item[$\bullet$]
$f\in \mathcal{W}'_{P,d}(\Delta)$ if and only if $f$ has a (possibly non-isolated) critical point at $\mathbf{0}$, satisfies $\Gamma(f)=\Delta$,  and is Newton non-degenerate on every proper face of $\Delta$, while it may fail to be Newton non-degenerate on $\Delta$ itself.
\end{enumerate}
\emph{Hereafter, we shall assume that $\mathcal{W}_{P,d}(\Delta)\not=\emptyset$.} Then $\mathcal{W}'_{P,d}(\Delta)\subseteq\overline{\mathcal{W}_{P,d}(\Delta)}$, where the bar stands for the closure of $\mathcal{W}_{P,d}(\Delta)$ in $\mathcal{W}_{P,d}$.  In particular, this implies the following.

\begin{remark}\label{rem-singset}
For any $f\in\mathcal{W}'_{P,d}(\Delta)$, the affine surface $V(f)$ has no singularity on $\{xyz=0\}$, except possibly at the origin, that is, $\mbox{Sing}(V(f))\subseteq\mathbb{C}^{*3}\cup\{\mathbf{0}\}$. 
Indeed, assume, for instance, that the line $x=y=0$ is contained in $\mbox{Sing}(V(f))$. Take $f'\in\mathcal{W}_{P,d}(\Delta)$ sufficiently close to $f$. Clearly, just like for $f$, the polynomial $f'$ does not contained any term of the form $z^\gamma$. Now, since $f'$ has an isolated singularity at $\mathbf{0}$, it must include a monomial of the form $xz^\alpha$ or $yz^\beta$, and necessarily such a monomial must be a vertex of $\Delta$. But then $f$ also has such a monomial, and hence $\partial f/\partial x (0,0,z)\not=0$, which is a contradiction. The other cases follow by similar arguments.
\end{remark}

Remark \ref{rem-singset} implies that for any $f\in\mathcal{W}'_{P,d}(\Delta)$, the singularities of the weighted projective variety $V_P(f)$ located in the coordinate subspaces of $\mathbb{P}^2(P)$ may arise solely from the singularities of the ambient space $\mathbb{P}^2(P)$ itself. In other words, the points in $\mbox{Sing}(V_P(f)) \cap \{xyz = 0\}$ are \emph{non-critical} points of $f$ (i.e., they are not zeros of the Jacobian of $f$), but only reflect singularities inherited from $\mathbb{P}^2(P)$.

Finally, note that since the elements of $\mathcal{W}'_{P,d}(\Delta)$ have at most a finite number of $1$-dimensional critical sets---which are orbits through the origin of the associated $\mathbb{C}^*$-action with respect to $P$---they are special types of so-called \emph{weakly almost Newton non-degenerate functions} (see \cite[\S 2.7]{O2}).

\subsection{Main result of this section}
It concerns the Milnor number of a generic plane section of an element in $\mathcal{W}'_{P,d}(\Delta)$, and is stated as follows.

\begin{theorem}\label{mt1}
Let $f\in\mathcal{W}'_{P,d}(\Delta)$, and let $H=\{(x,y,z)\in\mathbb{C}^3\, ;\, z=ax+by\}$ be a generic plane for $V(f)$ through the origin. 
\begin{enumerate}
\item
If $p_1=p_2=p_3$ or $p_1\geq p_2>p_3$, then the following properties hold true:
\begin{enumerate} 
\item
the polynomial $f^H(x,y):=f\vert_H(x,y,ax+by)$ is Newton non-degenerate with respect to the coordinates $(x,y)$; 
\item
the \emph{Newton number}\footnote{For the definition of the Newton number $\nu(f^H)$ of $f^H$---which itself depends on the definition of the Newton number of a compact polyhedron given in \cite[\S 1.7]{K}---refer to \cite[\S 1.8]{K} in the case where $f^H$ is convenient, and to \cite[\S 2]{BO} otherwise.} $\nu(f^H)$ of $f^H$ is independent of $H$, and
the Milnor number $\mu(f^H)$ of $f^H$ at~$(0,0)$---which is equal to $\nu(f^H)$ by (a) above\footnote{This follows from \cite[\S 1.10, Th\'eor\`eme I]{K} if $f^H$ is convenient, and from \cite[Corollary 3.10]{BO} if $f^H$ is not convenient.}---is completely determined by~$\Delta$.
\end{enumerate}
\item
In the case where $p_1>p_2=p_3$, the same properties (a) and (b) still hold true if we assume further that $\Gamma(f^H)=\Gamma(f(x,y,0))$.
\end{enumerate}
\end{theorem}

Note that, by Remark \ref{rem-singset}, if $f\in\mathcal{W}'_{P,d}(\Delta)$ then $\mbox{Sing}(V(f))\subseteq\mathbb{C}^{*3}\cup\{\mathbf{0}\}$, and this implies that the polynomial $f(x,y,0)$ of the variables $x,y$ does not identically vanish, so that $\Gamma(f(x,y,0))\not=\emptyset$.

Theorem \ref{mt1} is proved in section \ref{proofmt1}.

\begin{remark}
The assumption of Newton non-degeneracy on the boundary of $\Delta$ is essential for the validity of the theorem. Without this assumption, the assertion fails (see \cite[Theorem 2]{O4}).
\end{remark}

\section{Newton weighted-L\^e-Yomdin polynomials}\label{nwlyp}

In this section, we introduce the principal object of study in this paper: the \emph{Newton weighted-L\^e-Yomdin polynomials}.

\subsection{The set $\mathcal{W}'_{P,d}(\Delta,\Xi)$}\label{sect-xilist}
Fix a list $\Xi=\{\xi_1,\ldots,\xi_k\}$ of isomorphism classes of reduced plane curve singularities,\footnote{Here, two germs of reduced plane curves are in the same isomorphism class if there is a local, ambient homeomorphism of $\mathbb{C}^2$ sending the first germ onto the second one.}  and associated with this list, define a subset 
\begin{equation*}
\mathcal{W}'_{P,d}(\Delta,\Xi)\subseteq \mathcal{W}'_{P,d}(\Delta)
\end{equation*}
as follows. Given a polynomial $f\in \mathcal{W}'_{P,d}(\Delta)$, take a regular simplicial cone subdivision $\Sigma^*$ of its dual Newton diagram $\Gamma^*(f)$, and consider the toric modification $\pi\colon X\to \mathbb{C}^3$ associated with this subdivision. (For the definitions, see \cite[Chap.~II, \S 1 and Chap.~III,~\S 3]{O1}.) Note that $P$ is a vertex of $\Sigma^*$. Then, pick a cone $\sigma:=\mbox{Cone}(P,Q,R)$ of $\Sigma^*$ which includes $P$ as a vertex, and consider the corresponding toric chart $\mathbb{C}^3_\sigma$ of $X$ with coordinates $(u,v,w)$. Let $\hat E(P)$, $\hat E(Q)$ and $\hat E(R)$ denote the divisors of $\pi$ associated with the vertices $P$, $Q$ and $R$ respectively (see ibid.), and let $\widetilde{V}(f)$ be the strict transform of $V(f)$ by $\pi$. Assume that $\hat E(P)$, $\hat E(Q)$ and $\hat E(R)$ are defined by $u=0$, $v=0$ and $w=0$, respectively. Finally, put 
\begin{equation*}
E(P,f):=\hat E(P)\cap \widetilde{V}(f)
\quad\mbox{and}\quad
\hat E(P)^*:=\hat E(P)\setminus (\hat E(Q)\cup \hat E(R)),
\end{equation*}
and write $\mbox{Sing}(E(P,f))$ for the singular locus of $E(P,f)$. By \cite[\S 3]{O3}, $\mbox{Sing}(E(P,f))$ is a finite set, and by \cite[Lemma 3.1]{EO3}, we have
\begin{equation}\label{cincl2}
\mbox{Sing}(E(P,f))\subseteq \hat E(P)^*.
\end{equation}
Moreover, by \cite[Proposition 3.2]{EO3} and the comment that immediately follows it, there exists an analytic isomorphism
\begin{equation}\label{analyiso}
\psi\colon \hat E(P)^* \to\mathbb{P}^{2*}(P)
\end{equation}
such that 
\begin{equation*}
\psi(E(P,f)\cap\hat E(P)^*)=V_P(f)\cap\mathbb{P}^{2*}(P).
\end{equation*}
In particular, if $\rho_i\in\mbox{Sing}(V_P(f))\cap \mathbb{P}^{2*}(P)$ and if $\varrho_i:=\psi^{-1}(\rho_i)$ is the corresponding singular point in $\mbox{Sing}(E(P,f))\subseteq \hat E(P)^*$, then the germs
\begin{equation*}
(V_P(f),\rho_i) 
\quad\mbox{and}\quad
(E(P,f),\varrho_i)
\end{equation*}
 are topologically equivalent.\footnote{One must not confuse $\rho_i$ and $\varrho_i$, which, although topologically equivalent, lie in different spaces.}  The subset $\mathcal{W}'_{P,d}(\Delta,\Xi)$ of $\mathcal{W}'_{P,d}(\Delta)$ associated with the list of singularities $\Xi=\{\xi_1,\ldots,\xi_k\}$ is then defined as follows.

\begin{definition}\label{def-avecXI}
A polynomial $f\in\mathcal{W}'_{P,d}(\Delta)$ is in $\mathcal{W}'_{P,d}(\Delta,\Xi)$ if $V_P(f)$ has $k$ isolated singular points $\rho_1,\ldots,\rho_k$ in $\mathbb{P}^{2*}(P)$ that produce $k$ isolated singularities $\varrho_1,\ldots,\varrho_k$ of $E(P,f)$ in $\hat E(P)^*$ such that for any $1\leq i\leq k$ the isomorphism class of the (equivalent) germs $(V_P(f),\rho_i) $ and $(E(P,f),\varrho_i)$ is $\xi_i$.
\end{definition}

\subsection{The sets $\mathcal{W}''_{P,d,m}(\Delta)$ and $\mathcal{W}''_{P,d,m}(\Delta,\Xi)$ of weak-Newton weighted-L\^e-Yomdin polynomials}
Now, let $m>0$ be a positive integer, and let $\mathcal{W}''_{P,d,m}(\Delta)$ (respectively, $\mathcal{W}''_{P,d,m}(\Delta,\Xi)$) be the space of polynomials in the variables $x,y,z$ with an isolated critical point at $\mathbf{0}\in\mathbb{C}^3$ and of the form 
\begin{equation}\label{dews}
g=f+h, 
\end{equation}
where $f\in\mathcal{W}'_{P,d}(\Delta)$ (respectively, $f\in\mathcal{W}'_{P,d}(\Delta,\Xi)$) and $h$ is a polynomial such that:
\begin{enumerate}
\item
$\Gamma_+(h)\subseteq\Gamma_+(f)$;
\item
$d(P,h)=d+m$ (in particular, this implies that $f$ is the weighted homogeneous initial polynomial of $g$);
\item
$\mbox{Sing}(V(f))\cap V(h_P)=\{\mathbf{0}\}$, where $h_P$ is the weighted homogeneous initial polynomial of~$h$.
\end{enumerate} 
Here, $\Gamma_+(h)$ and $\Gamma_+(f)$ denote the Newton polyhedra of $h$ and $f$, respectively, and $d(P,h)$ is the minimal value of the map 
\begin{equation*}
(u,v,w)\in \Gamma_+(h)\subseteq\mathbb{R}^3_+\mapsto p_1u+p_2v+p_3w \in\mathbb{N}. 
\end{equation*}
To ensure that $\mathcal{W}''_{P,d,m}(\Delta)$ (respectively, $\mathcal{W}''_{P,d,m}(\Delta,\Xi)$) is finite-dimensional, we additionally assume that there is an arbitrary large, fixed integer $M\geq d+m$ such that for any element $g=f+h$ in $\mathcal{W}''_{P,d,m}(\Delta)$ (respectively, in $\mathcal{W}''_{P,d,m}(\Delta,\Xi)$), the degree of $h$ with respect to $P$ is less than or equal to $M$.

Note that, by~(2), the condition~(1) is equivalent to $\Gamma(g)=\Gamma(f)$ and it is always satisfied if $f$ is \emph{convenient} (i.e., if $\Gamma_+(f)$ meets each coordinate axis). Remark \ref{rem-singset} shows that the condition (3) is equivalent to the following one:
\begin{enumerate}
\item[($3'$)]
$\mbox{Sing}(V_P(f))\cap V_P(h_P)\cap\mathbb{P}^{2*}(P)=\emptyset$, 
where $\mbox{Sing}(V_P(f))$ is the singular locus of the weighted projective curve $V_P(f)$.
\end{enumerate} 
Also, observe that a polynomial $g=f+h$ in $\mathcal{W}''_{P,d,m}(\Delta)$ (respectively, in $\mathcal{W}''_{P,d,m}(\Delta,\Xi)$) always vanishes at $\mathbf{0}$ and has $f$ as face-function with respect to $P$.

\begin{example}
Let $g$ be a polynomial of the form $g=f+h$, where $f\in\mathcal{W}'_{P,d}(\Delta)$ is the weighted homogeneous initial polynomial of $g$ and $h$ is a polynomial such that $\Gamma_+(h)\subseteq\Gamma_+(f)$ and $d(P,h)=d+m$. If $g$ defines a \emph{weighted-L\^e-Yomdin singularity} at~$\mathbf{0}$ (i.e., if $\mbox{Sing}(V_P(f))\cap V_P(h_P)=\emptyset$, which is a slightly stronger condition than (3) and ($3'$)), then $g\in\mathcal{W}''_{P,d,m}(\Delta)$. 
\end{example}

\begin{example}
If $P$ is the trivial weight vector ${}^t(1,1,1)$ and if $g$ is a polynomial that defines a \emph{classical L\^e-Yomdin singularity} at $\mathbf{0}$ (i.e., if $\mbox{Sing}(V_P(g_d))\cap V_P(g_{d+m})=\emptyset$, where $g=g_d+g_{d+m}+\cdots$ is the homogeneous decomposition of $g$), then, modulo a linear change of coordinates, $f:=g_d$ and $h:=g_{d+m}+\cdots$ satisfy conditions (1), (2) and (3) above, and $f=g_d\in\mathcal{W}'_{P,d}(\Gamma(g_d))$, that is, $g\in\mathcal{W}''_{P,d,m}(\Gamma(g_d))$.
\end{example}

These examples motivate the following definition.

\begin{definition}\label{mdef}
The polynomials in the sets $\mathcal{W}''_{P,d,m}(\Delta)$ and $\mathcal{W}''_{P,d,m}(\Delta,\Xi)$ will be referred to as \emph{weak-Newton weighted-L\^e-Yomdin polynomials}.
\end{definition}

Here, the term ``weak-Newton'' refers to the Newton non-degeneracy condition (relative to the proper faces of $\Delta$) imposed on the initial weighted homogeneous polynomial $f$ of  any element $g=f+h$ in $\mathcal{W}''_{P,d,m}(\Delta)$ or $\mathcal{W}''_{P,d,m}(\Delta,\Xi)$.

\subsection{The set $\mathcal{W}''_{P,d,m}(\Delta,\Xi)_{\textnormal{ND}}$ of Newton weighted-L\^e-Yomdin polynomials}
In the special case where the plane curve singularities defining the isomorphism classes $\xi_i$ in the list $\Xi$ are all given by \emph{Newton non-degenerate} functions, we want to consider those elements in $\mathcal{W}'_{P,d}(\Delta,\Xi)$ whose associated curves exhibit singularities not merely topologically characterized by those in the list $\Xi$, but which additionally satisfy a more stringent condition, as specified in Definition \ref{def-spaceND} below. To reflect this strengthened condition, we shall denote the resulting space by $\mathcal{W}'_{P,d}(\Delta,\Xi)_{\text{ND}}$, and the associated space of weak-Newton weighted-L\^e-Yomdin polynomials by $\mathcal{W}''_{P,d,m}(\Delta,\Xi)_{\text{ND}}$. (Here, the letters ``ND'' stand for ``Newton non-degenerate''.)

To make this precise, given $f\in \mathcal{W}'_{P,d}(\Delta,\Xi)$, take---just as in section \ref{sect-xilist}---a regular simplicial cone subdivision $\Sigma^*$ of the dual Newton diagram $\Gamma^*(f)$, and consider the toric modification $\pi\colon X\to \mathbb{C}^3$ associated with $\Sigma^*$. Pick a cone $\sigma:=\mbox{Cone}(P,Q,R)$ of $\Sigma^*$ which includes $P$ as a vertex, and consider the corresponding toric chart $\mathbb{C}^3_\sigma$ of $X$ with coordinates $(u,v,w)$. Finally, assume again that the divisors $\hat E(P)$, $\hat E(Q)$ and $\hat E(R)$ are defined by $u=0$, $v=0$ and $w=0$, respectively. Then the sets $\mathcal{W}'_{P,d}(\Delta,\Xi)_{\text{ND}}$ and $\mathcal{W}''_{P,d,m}(\Delta,\Xi)_{\text{ND}}$ are defined as follows.

\begin{definition}\label{def-spaceND}
Suppose that each $\xi_i \in \Xi$ is defined by a \emph{Newton non-degenerate} function $\varphi_i$.
\begin{enumerate}
\item
A polynomial $f\in\mathcal{W}'_{P,d}(\Delta,\Xi)$ is said to be in $\mathcal{W}'_{P,d}(\Delta,\Xi)_{\text{ND}}$ if for each $\varrho_i\in\mbox{Sing}(E(P,f))$ such that $(E(P,f),\varrho_i)\simeq\xi_i$ ($1\leq i\leq k$), there exists an admissible\footnote{By \emph{admissible coordinate chart}, we mean an analytic chart $(U_{\varrho_i},(u',v',w'))$ such that $u'=u$ and $(v',w')$ is an analytic coordinate change of $(v,w)$.} coordinate chart $U_{\varrho_i}$ of $X$ at $\varrho_i$, with local coordinates $(u,v',w')$, in which $\hat E(P)$ is defined by $u=0$ and the pull-back $\pi^*f$ of $f$ by $\pi$ is given by 
\begin{equation*}
\pi^*f(u,v',w')={u}^d\phi_i(v',w'), 
\end{equation*}
where  the defining function $\phi_i(v',w')$ of $E(P,f)$ is Newton non-degenerate and such that the Newton boundaries $\Gamma(\phi_i)$ and $\Gamma(\varphi_i)$ coincide.
\item 
A polynomial $g=f+h\in\mathcal{W}''_{P,d,m}(\Delta,\Xi)$ is defined to be in $\mathcal{W}''_{P,d,m}(\Delta,\Xi)_{\text{ND}}$ if its weighted homogeneous initial polynomial $f$ belongs to $\mathcal{W}'_{P,d}(\Delta,\Xi)_{\text{ND}}$. 
\item 
The polynomials $g$ in $\mathcal{W}''_{P,d,m}(\Delta,\Xi)_{\text{ND}}$ will be referred to as \emph{Newton weighted-L\^e-Yomdin polynomials}, and singularities defined by such polynomials will be called \emph{Newton weighted-L\^e-Yomdin singularities}.
\end{enumerate}
\end{definition}

\section{Constructibility of the sets $\mathcal{W}'_{P,d}(\Delta,\Xi)$ and $\mathcal{W}''_{P,d,m}(\Delta,\Xi)$}\label{sect-constructibility}

The result presented in this section will be useful for establishing the main theorems of this paper in subsequent sections.

It is easy to see that $\mathcal{W}'_{P,d}(\Delta)$ and $\mathcal{W}''_{P,d,m}(\Delta)$ are constructible sets. Hereafter, we show the more subtle assertion that the sets $\mathcal{W}'_{P,d}(\Delta,\Xi)$ and $\mathcal{W}''_{P,d,m}(\Delta,\Xi)$ are constructible too. To this end, we introduce two  auxiliary sets $\mathcal{W}'_{P,d}(\Delta,\bar\mu)$ and $\mathcal{W}''_{P,d,m}(\Delta,\bar\mu)$ defined via the following procedure.
For each $1\leq i\leq k$, let $\mu_i$ denote the Milnor number associated with the isomorphism class $\xi_i\in\Xi$, and let $\bar\mu$ be the (unordered) multi-set $\bar\mu:=\{\mu_1,\ldots,\mu_k\}$. In analogy with the definitions of $\mathcal{W}'_{P,d}(\Delta,\Xi)$ and $\mathcal{W}''_{P,d,m}(\Delta,\Xi)$, we define subsets 
\begin{equation*}
\mathcal{W}'_{P,d}(\Delta,\bar\mu)\subseteq \mathcal{W}'_{P,d}(\Delta)
\quad\mbox{and}\quad
\mathcal{W}''_{P,d,m}(\Delta,\bar\mu) \subseteq\mathcal{W}''_{P,d,m}(\Delta)
\end{equation*}
in the following way.

\begin{definition}
A polynomial $f\in\mathcal{W}'_{P,d}(\Delta)$ is in $\mathcal{W}'_{P,d}(\Delta,\bar\mu)$ if $V_P(f)$ has $k$ isolated singular points $\rho_1,\ldots,\rho_k$ in $\mathbb{P}^{2*}(P)$ that produce $k$ isolated singularities $\varrho_1,\ldots,\varrho_k$ of $E(P,f)$ in $\hat E(P)^*$ such that for any $1\leq i\leq k$ the Milnor number of the (equivalent) germs $(V_P(f),\rho_i) $ and $(E(P,f),\varrho_i)$ is $\mu_i$.

A polynomial $g=f+h\in \mathcal{W}''_{P,d,m}(\Delta)$ is in $\mathcal{W}''_{P,d,m}(\Delta,\bar\mu)$ if and only if $f$ is in $\mathcal{W}'_{P,d}(\Delta,\bar\mu)$.
\end{definition}

Note that, in general, $\mathcal{W}'_{P,d}(\Delta,\bar\mu)$ and $\mathcal{W}''_{P,d,m}(\Delta,\bar\mu)$ are larger sets than $\mathcal{W}'_{P,d}(\Delta,\Xi)$ and $\mathcal{W}''_{P,d,m}(\Delta,\Xi)$, respectively. 

The main result of this section is stated as follows.

\begin{proposition}\label{prop-constructible}
The sets $\mathcal{W}'_{P,d}(\Delta,\bar\mu)$ and $\mathcal{W}''_{P,d,m}(\Delta,\bar\mu)$ are constructible, and the subsets $\mathcal{W}'_{P,d}(\Delta,\Xi)$ and $\mathcal{W}''_{P,d,m}(\Delta,\Xi)$ are unions of path-connected components of $\mathcal{W}'_{P,d}(\Delta,\bar\mu)$ and $\mathcal{W}''_{P,d,m}(\Delta,\bar\mu)$, respectively. In particular, $\mathcal{W}'_{P,d}(\Delta,\Xi)$ and $\mathcal{W}''_{P,d,m}(\Delta,\Xi)$ are also constructible.
\end{proposition}

Proposition \ref{prop-constructible} is proved is section \ref{proof-prop-constructible}.

\section{$\mu$-Zariski pairs of surface singularities}\label{sect-wzpss}

In this section, using Newton weighted-L\^e-Yomdin polynomials, we construct new $\mu$-Zariski pairs of surface singularities (see Theorem \ref{mt3}). A pivotal component of this construction involves showing that for any $\mu$-constant piecewise complex-analytic family $\{g_s\}_{0\leq s\leq 1}$ of weak-Newton weighted-L\^e-Yomdin polynomials $g_s=f_s+h_s$, the singularities of the weighted projective curve $V_P(f_s)$ remain stable as the parameter $s$ varies (see Theorem \ref{mt2}). Another key step in the proof of Theorem \ref{mt3} is a variant of Theorem \ref{mt1}, wherein the set $\mathcal{W}'_{P,d}(\Delta)$ is replaced by $\mathcal{W}''_{P,d,m}(\Delta)$ (see Theorem \ref{remthm1}). Lastly, a crucial property also used in establishing Theorem \ref{mt3} is that elements of $\mathcal{W}'_{P,d}(\Delta,\Xi)$ lying in the same path-connected component cannot produce Zariski pair of curves in $\mathbb{P}^2(P)$ (see Proposition \ref{prop-samecc}).

\subsection{Zariski pairs of curves in weighted projective planes}\label{sect-zpcwpp}
Our approach to constructing $\mu$-Zariski pairs of surface singularities in $\mathbb{C}^3$ begins with the concept of Zariski pairs of curves in the weighted projective plane $\mathbb{P}^2(P)$. So, let us first define precisely what we mean by a Zariski pair of curves in $\mathbb{P}^2(P)$. 

For any non-empty subset $I\subseteq\{1,2,3\}$, define 
\begin{align*}
& \mathbb{P}^{I}(P):=\{[\zeta_1\!:\!\zeta_2\!:\!\zeta_3]_P\in \mathbb{P}^2(P)\, ;\, \zeta_i=0\mbox{ if } i\notin I\},\\
& \mathbb{P}^{*I}(P):=\{[\zeta_1\!:\!\zeta_2\!:\!\zeta_3]_P\in \mathbb{P}^2(P)\, ;\, \zeta_i=0\mbox{ if and only if } i\notin I\},
\end{align*}
where $\zeta_1:=x$, $\zeta_2:=y$ and $\zeta_3:=z$, and $[\zeta_1\!:\!\zeta_2\!:\!\zeta_3]_P$ denotes the equivalence class of $(\zeta_1,\zeta_2,\zeta_3)$ in $\mathbb{P}^2(P)$. Then pick $f_0$ and $f_1$ in $\mathcal{W}'_{P,d}(\Delta,\Xi)$, and for $j=0,1$ put 
\begin{equation*}
V_P^I(f_j):= V_P(f_j) \cap \mathbb{P}^{I}(P)
\quad\mbox{and}\quad
V_P^{*I}(f_j):= V_P(f_j) \cap \mathbb{P}^{*I}(P).
\end{equation*}
The collections of subsets $\mathbb{P}^{*I}(P)$ and $V_P^{*I}(f_j)$ for $I\subseteq\{1,2,3\}$, $I\not=\emptyset$, give toric stratifications of $\mathbb{P}^{2}(P)$ and $V_P(f_j)$, respectively:
\begin{equation}\label{toricstratifications}
\mathbb{P}^{2}(P)=\bigsqcup_{\emptyset\not=I\subseteq\{1,2,3\}}\mathbb{P}^{*I}(P)
\quad\mbox{and}\quad
V_P(f_j)=\bigsqcup_{\emptyset\not=I\subseteq\{1,2,3\}}V_P^{*I}(f_j).
\end{equation}

We introduce a class of homeomorphisms of pairs 
\begin{equation*}
\phi\colon (\mathbb{P}^{2}(P),V_P(f_0))\to (\mathbb{P}^{2}(P),V_P(f_1)), 
\end{equation*}
called \emph{admissible homeomorphisms}, as follows.

\begin{definition}
A homeomorphism of pairs $\phi\colon (\mathbb{P}^{2}(P),V_P(f_0))\to (\mathbb{P}^{2}(P),V_P(f_1))$ is called \emph{admissible} if for every non-empty subset $I\subseteq\{1,2,3\}$, we have 
\begin{equation*}
\phi(\mathbb{P}^{I}(P),V_P^I(f_0))=(\mathbb{P}^{I}(P),V_P^I(f_1)).
\end{equation*}
\end{definition}

In particular, if such a homeomorphism of pairs exists, then the restrictions
\begin{align*}
& \phi\colon (\mathbb{P}^{2*}(P),V_P^*(f_0))\to (\mathbb{P}^{2*}(P),V_P^*(f_1))\\
\mbox{and}\quad
& \phi\colon (U_z,U_z\cap V_P(f_0))\to (U_z,U_z\cap V_P(f_1))
\end{align*}
are homeomorphisms of pairs too, where 
\begin{align*}
& \mathbb{P}^{2*}(P):= \mathbb{P}^{2}(P)\setminus \{xyz=0\} = \mathbb{P}^{\{1,2,3\}}(P),\\
& V_P^*(f_j):=\mathbb{P}^{2*}(P)\cap V_P(f_j)=V_P^{\{1,2,3\}}(f_j),\\
\mbox{and}\quad & U_z:=\mathbb{P}^{2}(P)\setminus \{z=0\}=\bigcup_{3\in I\subseteq\{1,2,3\}}\mathbb{P}^{*I}(P).
\end{align*}
Here, note that if $p_3=1$, then $U_z$ is the affine coordinate chart associated to $z$ (see section \ref{sect-wpswpc}). Within the preceding notation, care must be taken not to confuse the space $\mathbb{P}^{2*}(P)$, introduced in \eqref{def-p2s}, with the space 
\begin{align*}
\mathbb{P}^{*\{2\}}(P)=\{[x\!:\!y\!:\!z]_P\in \mathbb{P}^2(P)\, ;\, x=z=0,\,y\not=0\}.
\end{align*}

We can now precisely formulate the definition of a Zariski pair of curves in $\mathbb{P}^2(P)$ as follows.

\begin{definition}\label{def-zpwpc}
We say that two elements $f_0$ and $f_1$ in $\mathcal{W}'_{P,d}(\Delta,\Xi)$ form a \emph{Zariski pair of curves in the weighted projective plane $\mathbb{P}^2(P)$} if the following two conditions are satisfied:
\begin{enumerate}
\item
there exist regular neighbourhoods $N_0$ and $N_1$ of $V_P(f_0)$ and $V_P(f_1)$, respectively, such that $(N_0,V_P(f_0))$ and $(N_1,V_P(f_1))$ are homeomorphic (equivalently, $V_P(f_0)$ and $V_P(f_1)$ have the same ``combinatoric'' in the sense of \cite{Artal,ACT});
\item
there is no admissible homeomorphism between the pairs $(\mathbb{P}^{2}(P),V_P(f_0))$ and $(\mathbb{P}^{2}(P),V_P(f_1))$.
\end{enumerate}
\end{definition}

Note that being a Zariski pair of curves in $\mathbb{P}^2(P)$ is a \emph{global} notion.

\begin{example}
If $f_0,f_1\in\mathcal{W}'_{P,d}(\Delta,\Xi)$ have the same combinatoric and if the fundamental groups $\pi_1(\mathbb{P}^{*I}(P)\setminus V_P^{*I}(f_0))$ and $\pi_1(\mathbb{P}^{*I}(P)\setminus V_P^{*I}(f_1))$ are not isomorphic for some non-empty subset $I\subseteq \{1,2,3\}$---in particular, if  $\pi_1(\mathbb{P}^{2*}(P)\setminus V_P^*(f_0))\not\simeq \pi_1(\mathbb{P}^{2*}(P)\setminus V_P^*(f_1))$ or if $\pi_1(U_z,U_z\cap V_P(f_0))\not\simeq\pi_1(U_z,U_z\cap V_P(f_1))$---then $f_0$ and $f_1$ form a Zariski pair of curves in $\mathbb{P}^2(P)$. 
\end{example}

\begin{remark}
In the special case where $P$ is the trivial weight vector ${}^t(1,1,1)$, so that $\mathbb{P}^2(P)=\mathbb{P}^2$, the definition of Zariski pairs of curves presented above is more restrictive than the classical definition given in \cite{Artal}, which, beyond satisfying condition (1), requires only the non-existence of a homeomorphism between the pairs $(\mathbb{P}^{2}(P),V_P(f_0))$ and $(\mathbb{P}^{2}(P),V_P(f_1))$. In other words, for $P={}^t(1,1,1)$ the classical notion as defined in \cite{Artal} encompasses a broader class than the one considered in this work. Hereafter, unless explicitly stated otherwise, any reference to Zariski pairs in the standard projective plane $\mathbb{P}^2$ is to be understood in the usual sense of \cite{Artal}. By contrast, when $P\not={}^t(1,1,1)$, it is Definition \ref{def-zpwpc} that shall prevail.
\end{remark}

We also have the following important proposition, which plays a critical role in the proof of Theorem \ref{mt3} about $\mu$-Zariski pairs.

\begin{proposition}\label{prop-samecc}
Assume that $f_0$ and $f_1$ are two polynomials belonging to the same path-connected component of $\mathcal{W}'_{P,d}(\Delta,\Xi)$. Then there exists an admissible homeomorphism between the pairs $(\mathbb{P}^{2}(P),V_P(f_0))$ and $(\mathbb{P}^{2}(P),V_P(f_1))$.
\end{proposition}

This proposition underscores the relevance of our definition of Zariski pairs of curves in $\mathbb{P}^2(P)$,  in the sense that if the space $\mathcal{W}'_{P,d}(\Delta,\Xi)$ is path-connected, then it does not produce any Zariski pair---an usual property in the classical setting of Zariski pairs in $\mathbb{P}^2$, as defined in \cite{Artal}.

Proposition \ref{prop-samecc} is proved in section \ref{proof-prop-samecc}.

\subsection{From Zariski pairs of curves in $\mathbb{P}^2$ to those in $\mathbb{P}^2(P)$}

In this section, we explain a method that, for each classical \emph{strong} Zariski pair of curves in the standard projective space, associates a corresponding  Zariski pair in a weighted projective plane of special type (see Theorem \ref{mt4bis} and Corollary \ref{mt4}). Essentially, this construction comes from \cite{O5}.

Let $h_0(x,y,z)$ and $h_1(x,y,z)$ be two reduced homogeneous polynomials of degree $d$, and let $V_{P_0}(h_0)$ and $V_{P_0}(h_1)$ be the corresponding curves in the standard projective plane $\mathbb{P}^2(P_0)\equiv\mathbb{P}^2$, where $P_0:={}^t(1,1,1)$ is the trivial weight vector.
By a linear change of coordinates, we may assume that for $j=0,1$, the singularities of $V_{P_0}(h_j)$ are not located on $xyz=0$, and the polynomials $h_j(0,y,z)$, $h_j(x,0,z)$ and $h_j(x,y,0)$ are convenient and Newton non-degenerate. Without loss of generality, we may also suppose that the line $y=0$ is generic with respect to $V_{P_0}(h_j)$.
Put 
\begin{equation*}
H_j(x,y):=h_j(x,y,1)
\quad\mbox{and}\quad
F_j(x,y):=H_j(x,y^{p_1}), 
\end{equation*}
and consider the weighted homogenization $f_j$ of $F_j$ with respect to the weight vector $P:={}^t(p_1,1,1)$, namely:
\begin{equation*}
f_j(x,y,z):=z^{p_1d}\cdot F_j\Big(\frac{x}{z^{p_1}},\frac{y}{z}\Big),
\end{equation*}
which has $P$-degree $p_1d$. 
Note that $p_3=1$, and therefore $U_z:=\mathbb{P}^2(P)\setminus\{z=0\}$ is an affine coordinate chart (see section \ref{sect-wpswpc}). Also, observe that $\Gamma(f_0)=\Gamma(f_1)$ has a unique $2$-dimensional face $\Delta$, and the singularities  of the weighted projective curves $V_P(f_0)$ and $V_P(f_1)$ are completely determined by those of the projective curve $V_{P_0}(h_0)$ and $V_{P_0}(h_1)$. More precisely, each singular point of $V_{P_0}(h_0)$ and $V_{P_0}(h_1)$ is duplicated $p_1$ times in $V_P(f_0)$ and $V_P(f_1)$, respectively. 
The curves $V_P(f_0)$ and $V_P(f_1)$ may inherit quotient singularities from the ambient weighted projective plane $\mathbb{P}^2(P)$, even if their defining equations are locally regular at these points. Nevertheless, since both $V_{P_0}(h_0)$ and $V_{P_0}(h_1)$ are non-singular along the locus $xyz=0$ in $\mathbb{P}^2$, the singularities induced on $V_P(f_0)$ and $V_P(f_1)$ by the singular points of $\mathbb{P}^2(P)$ are equivalent.
Finally, as noted in section \ref{subsect-wh}, the above assumptions imply that $f_0$ and $f_1$ are Newton non-degenerate on every proper face of $\Delta$. By combining these observations with \cite[Theorem (3.4)]{O5}, we arrive at the following theorem, which is essentially Theorem (5.8) of \cite{O5}.

\begin{theorem}\label{mt4bis}
Assume that $h_0$ and $h_1$ form a Zariski pair of curves in the standard projective plane $\mathbb{P}^2$. Take a generic line $L_\infty$ for both $V_{P_0}(h_0)$ and $V_{P_0}(h_1)$, and put $\mathbb{C}^2:=\mathbb{P}^2\setminus L_\infty$. If the fundamental groups $\pi_1(\mathbb{C}^2\setminus V_{P_0}(h_0))$ and $\pi_1(\mathbb{C}^2\setminus V_{P_0}(h_1))$ are not isomorphic---so that, by definition, $h_0$ and $h_1$ form a \emph{strong generic affine} Zariski pair of curves in $\mathbb{P}^2$ (see \cite[\S 5 (C)]{O5})---then $f_0$ and $f_1$ form a Zariski pair of curves in the weighted projective plane $\mathbb{P}^2(P)$, where $P:={}^t(p_1,1,1)$.
\end{theorem}

\begin{proof}
As mentioned earlier, $f_0$ and $f_1$ have the same Newton boundary $\Delta$, and are Newton non-degenerate on the proper faces of $\Delta$. Moreover, since $V_{P_0}(h_0)$ and $V_{P_0}(h_1)$ share the same combinatoric, it follows from the above observations that $V_P(f_0)$ and $V_P(f_1)$ have the same combinatoric as well. In particular, $f_0$ and $f_1$ lie in $\mathcal{W}'_{P,p_1d}(\Delta,\Xi)$, where $\Xi$ is the set of isomorphism classes of the singularities of $V_P(f_0)$ and $V_P(f_1)$ that are located in $\mathbb{P}^{2*}(P)$. Now, by \cite[Theorem (3.4)]{O5}, we have
\begin{equation*}
\pi_1(U_z\setminus V_P(f_j))\simeq\pi_1(\mathbb{C}^2\setminus V_{P_0}(h_j))
\end{equation*}
for $j=0,1$. In particular, the pairs $(U_z,V_P(f_0))$ and $(U_z,V_P(f_1))$ are not homeomorphic, and therefore $f_0$ and $f_1$ form a Zariski pair of curves in $\mathbb{P}^2(P)$.
\end{proof}

Combined with \cite[Proposition (5.7)]{O5}, Theorem \ref{mt4bis} has the following corollary in case where $h_0$ and $h_1$ are \emph{irreducible}.

\begin{corollary}\label{mt4}
Assume that $h_0$ and $h_1$ are irreducible and form a Zariski pair of curves in $\mathbb{P}^2$ such that the embedded topologies of $V_{P_0}(h_0)$ and $V_{P_0}(h_1)$ are distinguished by the fundamental groups of their complements, that is, $\pi_1(\mathbb{P}^2\setminus V_{P_0}(h_0))$ and $\pi_1(\mathbb{P}^2\setminus V_{P_0}(h_1))$ are not isomorphic. (Such a Zariski pair is called a \emph{strong} Zariski pair, see \cite[\S 5 (C)]{O5}.)
Take a generic line $L_\infty$ for both $V_{P_0}(h_0)$ and $V_{P_0}(h_1)$, and again put $\mathbb{C}^2:=\mathbb{P}^2\setminus L_\infty$. If either $\pi_1(\mathbb{C}^2\setminus V_{P_0}(h_0))$ or $\pi_1(\mathbb{C}^2\setminus V_{P_0}(h_1))$ satisfies the \emph{(H.I.C.)-condition}\footnote{Following \cite[\S 3 (E)]{O5}, we say that a group $G$ satisfies the \emph{Homological Injectivity Condition of the centre} (or (H.I.C.)-condition in short) if $Z(G)\cap D(G)=\{e\}$, where $Z(G)$, $D(G)$ and $e$ are the centre, the commutator subgroup and the trivial element of the group $G$, respectively.}, then $f_0$ and $f_1$ form a Zariski pair of curves in the weighted projective plane $\mathbb{P}^2(P)$, where $P:={}^t(p_1,1,1)$.
\end{corollary}

\begin{proof}
As previously noted, $f_0$ and $f_1$ lie in $\mathcal{W}'_{P,p_1d}(\Delta,\Xi)$, where $\Xi$ is the set of isomorphism classes of the singularities of $V_P(f_0)$ and $V_P(f_1)$ located in $\mathbb{P}^{2*}(P)$, and $V_P(f_0)$ and $V_P(f_1)$ have the same combinatoric. Now, as $h_0$ and $h_1$ form a \emph{strong} Zariski pair of \emph{irreducible} curves in $\mathbb{P}^2$, and since either $\pi_1(\mathbb{C}^2\setminus V_{P_0}(h_0))$ or $\pi_1(\mathbb{C}^2\setminus V_{P_0}(h_1))$ satisfies the (H.I.C.)-condition, it follows from \cite[Proposition~(5.7)]{O5} that $h_0$ and $h_1$ form a \emph{strong generic affine} Zariski pair, so that
\begin{equation*}
\pi_1(\mathbb{C}^2\setminus V_{P_0}(h_0))\not\simeq\pi_1(\mathbb{C}^2\setminus V_{P_0}(h_1)).
\end{equation*}
We conclude with Theorem \ref{mt4bis}.
\end{proof}

\begin{example}
Consider Zariski's original \emph{Zariski pair} of projective curves in $\mathbb{P}^2$, which is given by homogeneous polynomials $h_0$ and $h_1$ of degree $6$ having $6$ cusp singularities (see \cite{Z}). Here, $h_0$ is of ``torus type'' while $h_1$ is not, so that $\pi_1(\mathbb{C}^2\setminus V_{P_0}(h_0))\simeq\mathbb{B}_3$ (the braid group on 3 strings) while  $\pi_1(\mathbb{C}^2\setminus V_{P_0}(h_1))\simeq\mathbb{Z}$ (see \cite{Z,O6}). Consider the weight vector $P:={}^t(2,1,1)$, and as above, for $j=0,1$ put $H_j(x,y):=h_j(x,y,1)$ and $F_j(x,y):=H_j(x,y^2)$. Finally, let $f_0$ and $f_1$ denote the weighted homogenizations with respect to $P$ of $F_0$ and $F_1$, respectively. Then, by \cite[Theorem (5.8)]{O5} or Theorem \ref{mt4bis} above, $V_P(f_0)$ and $V_P(f_1)$ are weighted projective curves of $P$-degree $12$ with $12$ cusp singularities in $\mathbb{P}^{2*}(P)$, they have the same combinatoric, and $\pi_1(U_z\setminus V_P(f_0))\simeq\mathbb{B}_3$ while $\pi_1(U_z\setminus V_P(f_1))\simeq\mathbb{Z}$. In particular, $f_0$ and $f_1$ form a Zariski pair of curves in the weighted projective plane $\mathbb{P}^2(P)$.
\end{example}

\subsection{$\mu$-Zariski pairs of Newton weighted-L\^e-Yomdin singularities}

We now arrive at the central definition of this paper.

\begin{definition}\label{defwzpss}
We say that (the germs at $\mathbf{0}$ of) two weak-Newton weighted-L\^e-Yomdin polynomials $g_0$ and $g_1$ in $\mathcal{W}''_{P,d,m}(\Delta,\Xi)$ form a \emph{$\mu$-Zariski pair of surface singularities in $\mathbb{C}^3$} if the following three conditions are satisfied:
\begin{enumerate}
\item[(i)]
$g_0$ and $g_1$ have the same monodromy zeta-function and the same Teissier $\mu^*$-invariant at~$\mathbf{0}$;
\item[(ii)]
the surface-germs $V(g_0)$ and $V(g_1)$ (equivalently, the links $K_{g_0}:=\mathbb{S}_\varepsilon^5\cap V(g_0)$ and $K_{g_1}:=\mathbb{S}_\varepsilon^5\cap V(g_1)$) in $\mathbb{C}^3$ are homeomorphic;
\item[(iii)]
$g_0$ and $g_1$ cannot be joined by a $\mu$-constant continuous path in $\mathcal{W}''_{P,d,m}(\Delta)$.
\end{enumerate} 
\end{definition}

We recall that the Teissier $\mu^*$-invariant of $g_j$ ($j = 0, 1$) is the triple composed of the Milnor numbers at $\mathbf{0}$ of $g_j$ and of its generic plane section, together with the multiplicity of $g_j$ at~$\mathbf{0}$ (see \cite{Teissier2}).

We emphasize that the path involved in condition (iii) lies in $\mathcal{W}''_{P,d,m}(\Delta)\supseteq\mathcal{W}''_{P,d,m}(\Delta,\Xi)$. Also, we observe that if two elements $g_0,g_1\in\mathcal{W}''_{P,d,m}(\Delta,\Xi)$ cannot be joined by a continuous path in the $\mu$-constant stratum of $\mathcal{W}''_{P,d,m}(\Delta)$, then, of course, they cannot be joined by such a path in the $\mu^*$-constant stratum of $\mathcal{W}''_{P,d,m}(\Delta)$ either, and it is well known that this last condition is a necessary condition for the pairs of germs $(\mathbb{C}^3,V(g_0))$ and $(\mathbb{C}^3,V(g_1))$ to be non-homeomorphic (see \cite[Chap.~II, Th\'eor\`eme 3.9]{Teissier2}). In other words, being a $\mu$-Zariski pair of surface singularities is a necessary condition for being a Zariski pair of surface singularities, as defined in the introduction. Of course, a $\mu$-Zariski pair is also a $\mu^*$-Zariski pair. Here, $\mu^*$-Zariski pairs are defined similarly to $\mu$-Zariski pairs, with the only difference being the replacement of the term ``$\mu$-constant'' in item (iii) of Definition \ref{defwzpss} with the term ``$\mu^*$-constant''.

Studies on $\mu$-Zariski pairs of surface singularities have also been carried out in \cite{O2}, where a slightly different definition is adopted---specifically, condition (ii) is not required (see ibid., \S 3.2).  Pairs satisfying only conditions (i) and (ii) were examined in \cite{O3}, where they are referred to as \emph{Zariski pairs of links}.

\begin{remark}\label{rem-cpca}
Let us point out the following facts.
\begin{enumerate}
\item[(a)]
Two polynomials $g_0$ and $g_1$ cannot be joined by a $\mu$-constant continuous path in $\mathcal{W}''_{P,d,m}(\Delta)$, if and only if, they cannot be joined by a $\mu$-constant \emph{piecewise complex-analytic path}\footnote{\label{f14} By a \emph{piecewise complex-analytic path} in $\mathcal{W}''_{P,d,m}(\Delta)$,  we mean a continuous path $\gamma\colon s\in [0,1]\mapsto \gamma(s):=g_s\in\mathcal{W}''_{P,d,m}(\Delta)$ such that there  exists a partition $0=s_0<s_1<\cdots<s_{q_0}=1$ of the interval $[0,1]$, and for each $0\leq q\leq q_0-1$, there exists a complex-analytic function $\tilde \gamma_q$, defined on an open subset $U_q\subseteq\mathbb{C}$ containing $[s_q,s_{q+1}]$, satisfying $\tilde \gamma_q\vert_{[s_q,s_{q+1}]}=\gamma\vert_{[s_q,s_{q+1}]}$ (see \cite[\S 5.1]{EO2}).} in $\mathcal{W}''_{P,d,m}(\Delta)$.
Indeed, by \cite[section 3]{EO2}, the $\mu$-constant strata of isolated surface singularities are constructible sets, and taking the intersection with $\mathcal{W}''_{P,d,m}(\Delta)$, which is also constructible, does not change this structure. The stated property then follows from an argument similar to that used in the proof of \cite[Theorem 5.4 and Proposition 5.2]{EO2}.
\item[(b)]
By \cite[Theorem 3.2]{ABLMH} or \cite[Theorem 3.3]{EO3}, Newton weighted-L\^e-Yomdin polynomials in $\mathcal{W}''_{P,d,m}(\Delta,\Xi)_{\text{ND}}$ have the same Milnor number.  
Clearly, they also have the same multiplicity. Indeed, for any $g=f+h\in\mathcal{W}''_{P,d,m}(\Delta,\Xi)_{\text{ND}}$, we have $\Gamma(g)=\Gamma(f)=\Delta$, and hence the multiplicity of $g$ (which is defined as the minimum value of the sums $\nu_1+\nu_2+\nu_3$ for $(\nu_1,\nu_2,\nu_3)\in\Gamma(g)$) is completely determined by $\Delta$. 
\item[(c)]
Let $g_0=f_0+h_0$ and $g_1=f_1+h_1$  be two polynomials in $\mathcal{W}''_{P,d,m}(\Delta,\Xi)_{\text{ND}}$ such that the weighted projective curves $V_P(f_0)$ and $V_P(f_1)$ share the same combinatoric. Then, by a weighted homogeneous version of \cite[Theorem 25 and Remark 26]{O2}, 
the surface-germs $V(g_0)$ and $V(g_1)$ have the same dual resolution graph, and hence, the links $K_{g_0}$ and $K_{g_1}$ are diffeomorphic.
\end{enumerate}
\end{remark}

Note that unlike Zariski pairs of curves in $\mathbb{P}^2(P)$, being a $\mu$-Zariski pair of surface singularities in~$\mathbb{C}^3$ is a \emph{local} notion.

\subsection{Statements of the results}
The first result of this section precisely concerns $\mu$-Zariski pairs of surface singularities. It is stated as follows.

\begin{theorem}\label{mt3}
Assume that $f_0,f_1\in \mathcal{W}'_{P,d}(\Delta,\Xi)_{\text{\emph{ND}}}$ form a  Zariski pair of curves in $\mathbb{P}^2(P)$, and consider $g_0,g_1\in\mathcal{W}''_{P,d,m}(\Delta,\Xi)_{\text{\emph{ND}}}$ such that for $j=0,1$ the face-function of $g_j$ with respect to $P$ is $f_j$ (i.e., the decomposition~\eqref{dews} of $g_j$ as an element of $\mathcal{W}''_{P,d,m}(\Delta,\Xi)_{\text{\emph{ND}}}$ is of the form $g_j=f_j+h_j$). Then the following two statements hold true.
\begin{enumerate}
\item
If $p_1=p_2=p_3$ or $p_1\geq p_2>p_3$, then $g_0$ and $g_1$ form a $\mu$-Zariski pair of surface singularities in $\mathbb{C}^3$. 
\item
The same conclusion holds true for $p_1>p_2=p_3$ if we assume further that
\begin{equation*}
\Gamma(f_j^H)=\Gamma(f_j(x,y,0))
\end{equation*}
 for $j=0,1$, where again $H=\{(x,y,z)\in\mathbb{C}^3\, ;\, z=ax+by\}$ is a generic plane for $V(f_j)$ through the origin, and $f_j^H(x,y):=f_j\vert_H(x,y,ax+by)$.
\end{enumerate}
\end{theorem}

For the meaning of $\mathcal{W}'_{P,d}(\Delta,\Xi)_{\text{ND}}$ and $\mathcal{W}''_{P,d,m}(\Delta,\Xi)_{\text{ND}}$, see Definition \ref{def-spaceND} above. In particular, in this theorem, we assume that the plane curve singularities defining the isomorphism classes $\xi_i$ in the list $\Xi$ are all given by Newton non-degenerate functions.

Theorem \ref{mt3} shows that many examples of $\mu$-Zariski pairs of surface singularities in $\mathbb{C}^3$ can be derived from Zariski pairs of curves in a weighted projective space, and Theorem \ref{mt4bis} and Corollary \ref{mt4} show that these latter pairs can, in turn, be constructed from numerous known examples of classical Zariski pairs of curves in the standard projective plane.

Theorem \ref{mt3} is proved in section \ref{proofmt3}.
An important step in its proof  is the following result, which is a variant of Theorem \ref{mt1}, where the set $\mathcal{W}'_{P,d}(\Delta)$ is replaced by $\mathcal{W}''_{P,d,m}(\Delta)$.

\begin{theorem}\label{remthm1}
Let $g\in\mathcal{W}''_{P,d,m}(\Delta)$, and let $H:=\{z=ax+by\}\subseteq \mathbb{C}^3$ be a generic plane for $V(g)$ through the origin. 
\begin{enumerate}
\item
If $p_1=p_2=p_3$ or $p_1\geq p_2>p_3$, then the function $g^H(x,y):=g\vert_H(x,y,ax+by)$
is Newton non-degenerate, $\nu(g^H)$ is independent of $H$, and $\mu(g^H)$ is completely determined by~$\Delta$. 
\item
In the case where $p_1>p_2=p_3$, to get the same conclusions, we need (as in Theorem~3.3) the additional assumption $\Gamma(g^H)=\Gamma(f(x,y,0))$,
where $f$ is the weighted homogeneous initial polynomial of $g$.
\end{enumerate}
\end{theorem}

Theorem \ref{remthm1} is obtained by a straightforward modification of the argument used in the proof of Theorem \ref{mt1} (see section 6). The details are left to the reader.

The third result of this section, presented below, also plays a crucial role in establishing Theorem~\ref{mt3}. This result also has a clear and significant interest of its own.

\begin{theorem}\label{mt2}
If $\{g_s\}_{0\leq s\leq 1}$ is a $\mu$-constant piecewise complex-analytic family of weak-Newton weighted-L\^e-Yomdin polynomials $g_s\in\mathcal{W}''_{P,d,m}(\Delta)$ such that for $s=0$ we have $g_0\in\mathcal{W}''_{P,d,m}(\Delta,\Xi)$, then we also have $g_s\in\mathcal{W}''_{P,d,m}(\Delta,\Xi)$ for any $0\leq s\leq 1$. 
\end{theorem}

Informally, Theorem \ref{mt2} says that for any $\mu$-constant piecewise complex-analytic family of weak-Newton weighted-L\^e-Yomdin polynomials $g_s=f_s+h_s\in \mathcal{W}''_{P,d,m}(\Delta)$, the singularities of the weighted projective curves $V_P(f_s)$ remain stable as the deformation parameter $s$ varies. 

Note that, unlike Theorem \ref{mt3}, in Theorem \ref{mt2} we do not assume that the plane curve singularities $\xi_i\in\Xi$ are defined by Newton non-degenerate functions.

Theorems \ref{mt2} is proved in section \ref{proofmt2}.

\section{Proof of Proposition \ref{prop-constructible}}\label{proof-prop-constructible}

First, observe that the statements concerning 
$\mathcal{W}''_{P,d,m}(\Delta,\bar\mu)$
and
$\mathcal{W}''_{P,d,m}(\Delta,\Xi)$
follow immediately from those for 
$\mathcal{W}'_{P,d}(\Delta,\bar\mu)$
and
$\mathcal{W}'_{P,d}(\Delta,\Xi)$,
respectively. Also, note that the statement for $\mathcal{W}'_{P,d}(\Delta,\Xi)$ is an immediate consequence of the corresponding statement for  $\mathcal{W}'_{P,d}(\Delta,\bar\mu)$. So, we only have to show that $\mathcal{W}'_{P,d}(\Delta,\bar\mu)$ is constructible. As the projective plane is significantly more tractable than its weighted counterpart, we introduce an auxiliary set
$\check{\mathcal{W}}'_{P,d}(\Delta,\bar\mu)$
of \emph{homogeneous} polynomials as follows. 

To begin with, observe that the covering transformation $\mathbb{C}^3\to\mathbb{C}^3$ defined by
\begin{equation*}
(x,y,z)\mapsto (x^{p_1},y^{p_2},z^{p_3})
\end{equation*} 
induces a well-defined $p_1p_2p_3$-fold branched covering $\psi\colon \mathbb{P}^2\to\mathbb{P}^2(P)$ given by
\begin{equation}\label{mapphi}
[x\!:\!y\!:\!z]\mapsto [x^{p_1}\!:\!y^{p_2}\!:\!z^{p_3}]_P. 
\end{equation} 
Indeed, if $[x\!:\!y\!:\!z]=[x'\!:\!y'\!:\!z']$ in $\mathbb{P}^2$, then there exists $t\not=0$ such that $(x',y',z')=(tx,ty,tz)$, and therefore $[(x')^{p_1}\!:\!(y')^{p_2}\!:\!(z')^{p_3}]_P=[t^{p_1}x^{p_1}\!:\!t^{p_2}y^{p_2}\!:\!t^{p_3}z^{p_3}]_P=[x^{p_1}\!:\!y^{p_2}\!:\!z^{p_3}]_P$ in $\mathbb{P}^2(P)$. Here, $[...]$ and $[...]_P$ denote the equivalence classes in $\mathbb{P}^2$ and $\mathbb{P}^2(P)$, respectively.

Now, we say that a polynomial $F$ is in $\check{\mathcal{W}}'_{P,d}(\Delta,\bar\mu)$ if the following two conditions hold true:
\begin{enumerate}
\item
$F$ is the pull-back by $\psi$ of a polynomial $f\in\mathcal{W}'_{P,d}(\Delta,\bar\mu)$, that is, $F(x,y,z):=\psi^* f(x,y,z):=f(x^{p_1},y^{p_2},z^{p_3})$;
\item
for each $1\leq i\leq k$, the Milnor number $\mu(F,\check\rho_i)$ of the germ $(V_{P_0}(F),\check\rho_i)$ is $\mu_i$, where $P_0:={}^t(1,1,1)$ and $\check\rho_i$ is any of the $p_1 p_2 p_3$ singular points of $V_{P_0}(F)\subseteq\mathbb{P}^2$ that are sent to $\rho_i\in\mbox{Sing}(V_P(f))\cap\mathbb{P}^{2*}(P)$ by $\psi$.
\end{enumerate}

\begin{remark}
Note that if $F\in\check{\mathcal{W}}'_{P,d}(\Delta,\bar\mu)$, then it is Newton non-degenerate on every proper face of $\check\Delta:=\Gamma(F)$. This follows from the Newton non-degeneracy of $f$ on the proper faces of $\Delta$ (see \cite[Chap.~III, Proposition (1.3.3)]{O1}).
\end{remark}

Since $\psi$ is a $p_1p_2p_3$-fold branched covering, to prove that $\mathcal{W}'_{P,d}(\Delta,\bar\mu)$ is constructible, it suffices to show that $\check{\mathcal{W}}'_{P,d}(\Delta,\bar\mu)$ is constructible. 
Clearly, (1) is a constructible condition. To show that adding condition (2) keeps  constructibility, we proceed as follows. 

First, we recall the definition of a certain constructible set, denoted by $W(2,\delta,\mu)$, which we introduced in \cite[section 3]{EO2}. Let $\delta$ and $\mu$ be two integers with $\delta>0$. Consider the space $P(2,\delta)$ of polynomials in two variables and of degree less than or equal to $\delta$. This is a vector space of dimension $N:=\binom{2+\delta}{2}$, putting an order on the basis monomials $1=M_1(x,y)<M_2(x,y)<\cdots<M_N(x,y)$. Hereafter, we identify $P(2,\delta)$ with $\mathbb{C}^N$. Let 
\begin{equation*}
\pi_\delta\colon \mathcal{O}_2\equiv\mathbb{C}\{x,y\}\to P(2,\delta)
\end{equation*}
be the natural projection obtained by deleting the terms of degree greater than $\delta$. For any $f\in\mathcal{O}_2$, we write $J(f)$ for the Jacobian ideal of $f$, and we put 
\begin{equation*}
J_\delta(f):=\pi_\delta(J(f)).
\end{equation*} 
 An  element of $J_\delta(f)$ is written as $\pi_\delta(f_1 (\partial f/\partial x) +f_2 (\partial f/\partial y))$, where $f_1,f_2\in\mathcal{O}_2$, and we easily check that $J_\delta(f)$ is the subspace of $P(2,\delta)\equiv\mathbb{C}^N$ generated by the following set of $2N$ vectors:
\begin{equation*}
B_\delta(f):=\bigg\{\pi_\delta\bigg( M_1\,\frac{\partial f}{\partial x}\bigg),\ldots, \pi_\delta\bigg( M_N\,\frac{\partial f}{\partial x}\bigg), \pi_\delta\bigg( M_1\,\frac{\partial f}{\partial y}\bigg),\ldots,\pi_\delta\bigg( M_N\,\frac{\partial f}{\partial y}\bigg)\bigg\}.
\end{equation*} 
Hereafter, we identify $B_\delta(f)$ with an $N\times 2N$ matrix. 
Then  we consider the algebraic variety
\begin{align*}
V(2,\delta,\mu):=\{f \in P(2,\delta) \mid \mbox{all $(N-\mu)\times(N-\mu)$ 
minors of $B_\delta(f)$ vanish}\},
\end{align*}
and we define the constructible set $W(2,\delta,\mu)$ by
\begin{equation*}
W(2,\delta,\mu):=V(2,\delta,\mu-1)\setminus V(2,\delta,\mu).
\end{equation*}  
Note that for any $f\in P(2,\delta)$, the following equivalences hold true:
\begin{equation*}
f\in W(2,\delta,\mu)\Leftrightarrow \mbox{rk}\, B_\delta(f)=N-\mu \Leftrightarrow \dim P(2,\delta)/J_\delta(f)=\mu.
\end{equation*} 

Now, let $\eta$ denote the number of monomials in a generic polynomial $f$ belonging to $\mathcal{W}'_{P,d}(\Delta,\bar\mu)$, and let $c_1,\ldots,c_\eta$ be its coefficients. Let $F:=\psi^*f$, and let $(a_i,b_i)$ be the coordinates of $\check\rho_i$ in the affine chart $(u,v):= (x/z,y/z)$. Around $\check\rho_i$, we introduce the shifted coordinates $(u_i,v_i) := (u-a_i,v-b_i)$, and we consider $F$ as a polynomial in the variables $u_i,v_i$ and whose coefficients are in $c_1,\ldots,c_\eta,a_i,b_i$. 
To show that condition (2) is constructible, we first note that for $1 \leq i \leq k$, the condition $F(a_i, b_i) = 0$ is constructible. Moreover, the (local) condition $\mu(F, \check \rho_i) = \mu_i$, which is equivalent to 
\begin{equation}\label{i}
(c_1,\ldots,c_\eta,a_i,b_i)\in W(2,\delta,\mu_i)
\end{equation}
(for some $\delta > 0$ large enough) is also constructible. We now aim to formulate, as a condition on the parameters $c_1,\ldots,c_\eta,a_1,b_1,\ldots,a_k,b_k$, the (global) requirement that no distinct singular points $\check\rho_i\not=\check\rho_j$, $1\leq i\not=j\leq k$, merge into a single one. To this end, we consider an auxiliary function
\begin{equation*}
g(x,y,z):=f(x,y,z)+h(x,y,z),
\end{equation*}
where $h$ is a weighted homogeneous polynomial of $P$-degree $d+m$, with $m>0$ and $\mbox{Sing}(V_P(f))\cap V_P(h)=\emptyset$. For such a function $g$, the Milnor number can be effectively calculated. Indeed, since $f\in\mathcal{W}'_{P,d}(\Delta,\bar\mu)$, it follows from \cite[Theorem 3.2]{ABLMH} that $g$ has an isolated singularity at $\mathbf{0}$ with Milnor number
\begin{equation*}
\mu_+:=\mu(g)=(d/p_1-1)(d/p_2-1)(d/p_3-1)+m\mu^{\text{tot}},
\end{equation*}
where $\mu^{\text{tot}}$ is the sum of the local Milnor numbers at points of $\mbox{Sing}(V_P(f))$. As these Milnor numbers are independent of the choice of $f\in\mathcal{W}'_{P,d}(\Delta,\bar\mu)$, so is $\mu^{\text{tot}}$. Thus, ensuring that $\check\rho_i\not=\check\rho_j$ for $i\not=j$ is equivalent to assuming that
\begin{equation}\label{ii}
(c_1,\ldots,c_\eta,a_1,b_1,\ldots,a_k,b_k)\in W(3,\delta,\mu_+).
\end{equation}
Here, the set $W(3,\delta,\mu_+)$ is defined as $W(2,\delta,\mu)$, replacing $2$ by $3$ and $\mu$ by $\mu_+$.
Altogether now, consider the subset of $\mathbb{C}^\eta\times \mathbb{C}^{2k}$ consisting of points 
\begin{equation*}
(c_1,\ldots,c_\eta,a_1,b_1,\ldots,a_k,b_k)
\end{equation*}
 that satisfy the constructible conditions (1), \eqref{i} and \eqref{ii}. Then its projection on $\mathbb{C}^\eta$, which is nothing else than $\check{\mathcal{W}}'_{P,d}(\Delta,\bar\mu)$,  is a constructible set by Chevalley's theorem (see \cite[Th\'eor\`eme 3]{CC} and \cite[Th\'eor\`eme (1.8.4)]{Grothendieck}).

\section{Proof of Theorem \ref{mt1}}\label{proofmt1}

Throughout, we identify the integral points $(\alpha,\beta)$ of the Newton diagram $\Gamma(f^H)$ with the corresponding monomials $x^{\alpha}y^{\beta}$ in the polynomial $f^H$. 

If $p_1=p_2=p_3$, then necessarily $P={}^t(1,1,1)$. In this case, 
\begin{equation*}
f^H=c_0\prod_{i=1}^d(x+c_iy), 
\end{equation*}
where $c_0\not=0$ and $c_1,\ldots,c_d$ are mutually distinct. Indeed, since $f$ is reduced and $H$ is generic, $f^H$ has no multiple factor either. It follows that $f^H$ is Newton non-degenerate. Moreover, for $c_0$ and $c_i$ ($1\leq i\leq d$) generic, $f^H$ contains the terms $x^d$ and $y^d$ (up to coefficients), and hence, by Kouchnirenko's theorem \cite[\S 1.10, Th\'eor\`eme I]{K} or Milnor-Orlik's theorem \cite[Theorem~1]{MO}, we have
\begin{equation*}
\mu(f^H)=\nu(f)=d^2-2d+1.
\end{equation*}

From now on, suppose that the $p_i$'s are not all equal (i.e., $P\not={}^t(1,1,1)$). Expand $f$ with respect to the variable $z$ (which corresponds to the smallest weight $p_3$):
\begin{equation*}
f(x,y,z)=f(x,y,0)+\sum_{j=1}^q z^j\, h_j(x,y),
\end{equation*}
where $h_j(x,y)$ is a polynomial in $x,y$ of degree $d-p_3j$ with respect to the weight vector  ${}^t(p_1,p_2)$.
Clearly, since $p_1\geq p_2\geq p_3$, substituting $z=ax+by$ in $z^j\, h_j(x,y)$ produces only monomials of ${}^t(p_1,p_2)$-degree greater than or equal to $d+j(p_2-p_3)$. Put
\begin{equation*}
L:=\{(\nu_1,\nu_2)\in\mathbb{Z}^2_+\, ;\, \nu_1 p_1+\nu_2 p_2=d\}.
\end{equation*}
The Newton boundary $\Gamma(f(x,y,0))$ of $f(x,y,0)$ with respect to the coordinates $(x,y)$ defines a segment on the line $L$. Let us denote by $A=(a_1,a_2)$ and $B=(b_1,b_2)$ its ends. Assume, for instance, that $b_1\leq a_1$ and $a_2\leq b_2$ (see Figure \ref{figure1}). Note that $A=B$ is possible.

Now, we divide the proof into two cases according to $p_2>p_3$ or $p_2=p_3$.

\begin{figure}[t]
\includegraphics[width=16.5cm,height=4.5cm]{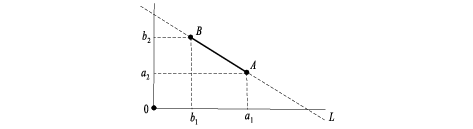}
\caption{The Newton boundary $\Gamma(f(x,y,0))$}
\label{figure1}
\end{figure}

\begin{figure}[b]
\includegraphics[width=16.5cm,height=7.5cm]{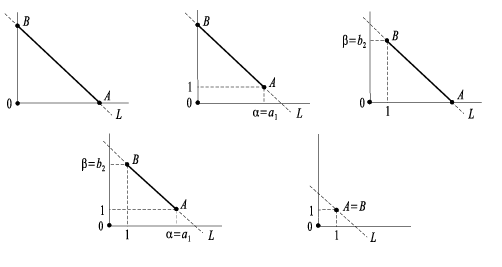}
\caption{The edge $AB$}
\label{figure2}
\end{figure}

\subsection{The case $p_2>p_3$}

In this case, since $d+j(p_2-p_3)>d$, the intersection $\Gamma(f^H)\cap L$ contains only the terms that come from $\Gamma(f(x,y,0))$. The terms produced by substituting $z=ax+by$ in $z^j\, h_j(x,y)$ are located strictly above $L$. Note that $\Gamma(f(x,y,0))$ (i.e., the face $AB$) is a proper face of $\Delta=\Gamma(f)$ and that on this face $f$ is Newton non-degenerate.

If $AB$ contains a term of the form $x^\alpha$ (or no term of this form but a monomial of   type $x^\alpha y$) and a term of the form $y^\beta$ (or, again, not of this form but of type $xy^\beta$), then $\alpha=a_1$ and $\beta=b_2$, and the edge $AB$ looks like one of the five pictures sketched in Figure~\ref{figure2}. In this case, the function $f^H$ is Newton non-degenerate and the Newton number $\nu(f^H)$ is entirely determined by $\Delta$. Indeed, the face-function $f_{AB}$ associated with the edge $AB$ is Newton non-degenerate by assumption, and if there is a point $C\in \Gamma(f^H)$ strictly above the line $L\supseteq AB$, then $C$ must be located on a coordinate axis, and hence the face-function corresponding to the edge $AC$ or $BC$ (according to the case) is Newton non-degenerate too and does not influence the Newton number of the face-function $f_{AB}$, that is, $\nu(f^H)=\nu(f_{AB})$. Thus, $f^H$ is Newton non-degenerate, and by \cite[\S 1.10, Th\'eor\`eme I]{K} (convenient case) or \cite[Corollary 3.10]{BO} (non-convenient case), we have $\mu(f^H)=\nu(f^H)$. 

Now, suppose that we are not in the above situation, and assume, for instance, that $A=x^{a_1}y^{a_2}$ is not of the form $x^{a_1}$ or $x^{a_1} y$. (We are only discussing what happens with the point $A$, as a completely symmetric discussion applies to the point $B$.) 

\begin{claim}\label{claimterm} 
Under this assumption, $f$ has a term of the form $x^\alpha z$. 
\end{claim}

\begin{proof}
We argue by contradiction. Suppose that $f$ does not include the terms $x^{a_1}$ and $x^{a_1} y$ and has no term of the form $x^\alpha z$.
As $\mathcal{W}_{P,d}(\Delta)\not=\emptyset$ and $f\in\mathcal{W}'_{P,d}(\Delta)\subseteq\overline{\mathcal{W}_{P,d}(\Delta)}$ (see section~\ref{sect-mngps-31}), there exists a sequence of polynomials $f_n\in\mathcal{W}_{P,d}(\Delta)$ approaching $f$ such that $f_n$ does not include the terms $x^{a_1}$ and $x^{a_1} y$. (Note that $f$ and $f_n$ have the same Newton boundary~$\Delta$.) If $f_n$ also contains no term of the form $x^\alpha z$, then  it does not have an isolated singularity at $\mathbf{0}$ (its singular locus contains the set defined by $y=z=0$), which is a contradiction with the fact that $f_n\in\mathcal{W}_{P,d}(\Delta)$.
Thus $f_n$ does include the term $x^\alpha z$, and necessarily the point $(\alpha,0,1)\equiv x^\alpha z$ lies on $\Delta$. In fact, $(\alpha,0,1)$ is a vertex of $\Delta$. (Indeed, for $(\alpha,0,1)$ to be an interior point of an edge, we would need to have two other points of $\Delta$ of the form $(\alpha',0,0)$ and $(\alpha'',0,\beta)$, with $\beta>1$, which, however, do not exist under our assumption.) Again, since $f$ and $f_n$ have the same Newton boundary $\Delta$, this implies that $f$ does indeed include the term $x^\alpha z$ with a non-zero coefficient.
\end{proof}

The monomial $x^\alpha z$ involved in Claim \ref{claimterm} comes from $z\, h_1(x,y)$, and substituting $z=ax+by$ into  it yields the terms $x^\alpha y$ and $x^{\alpha+1}$ in $f^H$. Put $C:=x^\alpha y\equiv (\alpha,1)$ and $C':=x^{\alpha+1}\equiv (\alpha+1,0)$. 
The presence of $C'$ in $f^H$ implies that $\Gamma(f^H)$ contains, up to a multiplicative constant, a term of the form $D=x^\beta$ with $\beta\leq\alpha+1$. 
Such a point $D$ comes from $f(x,0,z)$, which defines an edge $EF$ of $\Delta$ in the $xz$-coordinate plane. Put $E=x^{e_1}z^{d_1}$ and $F=x^{e_2}z^{d_2}$, and assume, for instance, $d_1<d_2$. Substituting $z=ax+by$ in $E$ and $F$ produces the terms $x^{e_1+d_1}$ and $x^{e_2+d_2}$ in $f^H$. Now, since $f$ is weighted homogeneous, 
$p_1e_1+p_3d_1=p_1e_2+p_3d_2$, 
and since $d_1<d_2$ and $p_1>p_3$, it follows that $e_1+d_1<e_2+d_2$, and $D=x^\beta$ comes from the point $E$ (i.e., $D=x^{e_1+d_1}$)---a point which is on the boundary $\partial \Delta$ of $\Delta$, so that $D$ is determined by $\partial \Delta$ and independent of $H$. In particular, the Newton number $\nu(f^H)$ of $f^H$ is independent of $H$, completely determined by $\Delta$.

\begin{figure}[t]
\includegraphics[width=16.5cm,height=7.5cm]{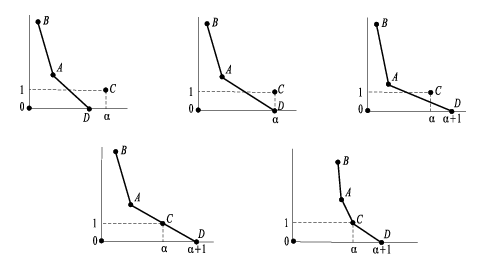}
\caption{The Newton boundary $\Gamma(f^H)$}
\label{figure3}
\end{figure}

The point $D$ may be below, on or above the line $AC$.  The Newton boundary of $f^H$ then looks like one of the five diagrams sketched in Figure~\ref{figure3}. (We only draw pictures in the case $A\not=B$, as the case $A=B$ can be handled in a similar manner.)
If $D$ is below the line $AC$, then $AD$ is an edge of $\Gamma(f^H)$ with no interior integral point, that is, it is of the form $c_0\, x^{a_1}y^{a_2}+c_1(a,b)\, x^{\beta}$, where the coefficient $c_1(a,b)$ depends linearly on $a$ and $b$, and therefore, it is Newton non-degenerate.
 If $D$ is on the line $AC$, then $D=C'$ and the face-function defined by the edge $ACC'$---which has no other interior integral point than the point $C$ itself\footnote{Indeed, the $P$-degree of $A\equiv x^{a_1}y^{a_2}$ and $x^\alpha z$ (which are monomials of $f$) is $d$, while the degree of $C\equiv x^\alpha y$ (which is a monomial of $f^H$) with respect to ${}^t(p_1,p_2)$ is $d+(p_2-p_3)$. The assertion follows from the fact that substituting $z=ax+by$ in $z^j\, h_j(x,y)$ produces only monomials of ${}^t(p_1,p_2)$-degree greater than or equal to $d+j(p_2-p_3)$.}---takes the form 
\begin{equation*}
c_0\, x^{a_1}y^{a_2}+c_1(a,b)\, x^\alpha y+c_2(a,b)\, x^{\alpha+1},
\end{equation*} 
and is Newton non-degenerate for generic values of $a$ and $b$. (Again, $c_1(a,b)$ and $c_2(a,b)$ are coefficients depending linearly on $a$ and $b$.)
Finally, if $D$ is strictly above the line $AC$, then, again, $D=C'$ and the face-functions associated with the edges $AC$ and $CC'$---which have no other interior integral point either---take the form 
\begin{equation*}
c_0\, x^{a_1}y^{a_2}+c_1(a,b)\, x^\alpha y
\quad\mbox{and}\quad
c_1(a,b)\, x^\alpha y+c_2(a,b)\, x^{\alpha+1}
\end{equation*} 
and are Newton non-degenerate. (Note that since $D$ is strictly located above the line $AC$, we have $a_1\not=\alpha a_2$.) 
In all cases, we see that $f^H$ is convenient, and we conclude that $\mu(f^H)=\nu(f^H)$ by Kouchnirenko's theorem (see \cite[\S 1.10, Th\'eor\`eme~I]{K}). 

\subsection{The case $p_2=p_3$}
As the $p_i$'s are not all equal, if $p_2=p_3$ then $p_1>p_2=p_3$.
In this case, substituting $z=ax+by$ in $z^j\, h_j(x,y)$ may add terms on the line $L$. These terms (which have $P$-degree $d$) come from $\sum_{j=1}^q b^j y^j h_j(x,y)$. (All the other terms have $P$-degree $>d$, and hence, are strictly above $L$.) In general, these terms may inflate the Newton boundary $\Gamma(f^H)$. However, under the assumption $\Gamma(f^H)=\Gamma(f(x,y,0))$, which says that all the added terms are in $\Gamma(f(x,y,0))$, this cannot occur. Then, since $f$ is Newton non-degenerate on the face $\Gamma(f(x,y,0))$, the restricted polynomial $f^H$ is Newton non-degenerate as well for generic values of $a$ and $b$, and we conclude again with \cite[\S 1.10, Th\'eor\`eme I]{K} or \cite[Corollary 3.10]{BO}.

\section{Proof of Theorem \ref{mt2}}\label{proofmt2}
 
Let $\Sigma^*$ be a regular simplicial cone subdivision of the dual Newton diagram $\Gamma^*(g_s)$ of $g_s$ such that $P$ is a vertex of $\Sigma^*$, and let $\pi\colon X\to \mathbb{C}^3$ be the toric modification associated with this subdivision. Take a cone $\sigma:=\mbox{Cone}(P,Q,R)$ of $\Sigma^*$ which includes $P$ as a vertex, and consider the corresponding toric chart $\mathbb{C}^3_{\sigma}$ of $X$. Finally, let $f_s\in\mathcal{W}'_{P,d}(\Delta)$ denote  the weighted homogeneous initial polynomial of $g_s$, and as in section \ref{sect-xilist}, let us first focus on the hypersurface $E(P,f_s)$. Then we have the following lemma.

\begin{lemma}\label{ctmn}
The total Milnor number $\mu^{\mbox{\tiny \emph{tot}}}(E(P,f_s))$ of $E(P,f_s)$ (i.e., the sum of the local Milnor numbers at the singular points $\varrho_s$ of $(E(P,f_s))$) is independent of $s$.
\end{lemma} 

\begin{proof}
As we showed in \cite[\S 4.1]{EO3}, under our assumptions the function $g_s$ is weakly almost Newton non-degenerate in the sense of \cite[\S 2.7]{O2}. Thus, by \cite[Theorem~9]{O2} or \cite[Lemma 3.2, Remark 3.6 and Theorem 3.7]{O3}, for each $s$ we have
\begin{equation}\label{okaformula}
\zeta_{g_s,\mathbf{0}}(t)=\zeta(t)\cdot (1-t^{d})^{\mu^{\text{\tiny tot}}(E(P,f_s))}\cdot\prod \zeta_{\pi^*\! g_s,\,\varrho_{s}}(t),
\end{equation}
where $\zeta_{g_s,\mathbf{0}}(t)$ is the monodromy zeta-function of $g_s$ at $\mathbf{0}$.
In this formula, $\zeta(t)$ stands for the monodromy zeta-function $\zeta_{f_{s,\tau},\mathbf{0}}(t)$ of $f_{s,\tau}$ at~$\mathbf{0}$ for $\tau\not=0$, where $\{f_{s,\tau}\}$ is a deformation of $f_s$ (i.e., for $\tau=0$, we have $f_{s,0}=f_s$) obtained from a small perturbation of the coefficients of $f_s$ so that $f_{s,\tau}$ is Newton non-degenerate and has an isolated singularity at the origin for all small non-zero value of the parameter $\tau$.\footnote{Such a deformation exists as $\mbox{Sing}(V(f_s))\subseteq \mathbb{C}^{*3}\cup \{\mathbf{0}\}$. Of course, the zeta-function $\zeta_{f_{s,\tau},\mathbf{0}}(t)$ is independent of $\tau$ for all small $\tau\not=0$. Also, note that by Varchenko's theorem \cite[Theorem~(4.1)]{V}, $\zeta(t)$ coincides with monodromy zeta-function of $g_\tau:=f_{s,\tau}+h_s$ ($\tau\not=0$) at~$\mathbf{0}$, so that the results of \cite{O2,O3} apply.} The product at the rightmost end of \eqref{okaformula} is taken over all the singular points $\varrho_{s}$ of the hypersurface $E(P,f_s)$. 
Finally, $\zeta_{\pi^*\!g_s,\,\varrho_{s}}(t)$ denotes the monodromy zeta-function of $\pi^*\! g_s$ at $\varrho_{s}$. 

By Milnor-Orlik's theorem \cite[Theorem 4]{MO}, the zeta-function $\zeta(t)$ that appears in the right-hand side of \eqref{okaformula} is completely determined by $P$ and $d$, and hence it does not depend on $s$.
Also, we easily check that for each singular point $\varrho_{s}$, the exceptional divisors of a suitable resolution of $\pi^*\! g_s$ at $\varrho_{s}$ all have multiplicity greater than $d$, so that, by A'Campo formula \cite[Th\'eor\`eme~3]{A}, the zeta-function $\zeta_{\pi^*\! g_s,\,\varrho_{s}}(t)$ is written as a product of the form $\prod_j (1-t^{m_j})^{\nu_j}$ with $m_j>d$. Now, since $\{g_s\}$ is $\mu$-constant, and hence $\zeta_{g_s,\mathbf{0}}(t)$ is independent of $s$ (see \cite[Lemma 12]{O2}), all this implies that the middle term 
\begin{equation*}
(1-t^{d})^{\mu^{\text{\tiny tot}}(E(P,f_s))}
\end{equation*}
in the right hand-side of \eqref{okaformula}, and hence the total Milnor number $\mu^{\text{\tiny tot}}(E(P,f_s))$, are both independent of $s$. 
\end{proof}

To complete the proof of the theorem, we now show that the family $\{f_s\}$ has no bifurcation of singularities in $\mathbb{P}^{2*}(P)$. We argue by contradiction. Suppose there exists $s_0$ such that $\{f_s\}$ has a bifurcation of singularities at a singular point $\rho_{s_0}$ of $V_P(f_{s_0})\cap \mathbb{P}^{2*}(P)$. This means that there is a small ball $B_\varepsilon(\rho_{s_0})\subseteq\mathbb{P}^{2*}(P)$ centred at $\rho_{s_0}$  such that $\rho_{s_0}$ is the only singular point of $V_P(f_{s_0})$ in $B_\varepsilon(\rho_{s_0})$ and it is either a ``newly born'' singularity or it is obtained as a merging of several singularities of $V_P(f_{s})$ for $s$ near $s_0$, $s\not=s_0$. In other words, $s_0$ is a parameter for which the
natural projection
\begin{equation*}
\{(a,s)\in\mathbb{P}^{2*}(P)\times [0,1]\, ;\, a\in \mbox{Sing}(V_P(f_s))\}\to [0,1]
\end{equation*}
fails to be a covering map. Then, by a theorem of Lazzeri \cite{Lazzeri} and L\^e \cite[Th\'eor\`eme~B]{Le} (see also \cite{H}), for $s\not=s_0$ near $s_0$ we have
\begin{equation}\label{sumle}
\sum_{\rho_s\in B_\varepsilon(\rho_{s_0})\cap \text{\tiny Sing}(V_P(f_s))}\mu(V_P(f_{s}),\rho_s) < \mu(V_P(f_{s_0}),\rho_{s_0}),
\end{equation}
where $\mu(V_P(f_{s}),\rho_s)$ is the local Milnor number of $V_P(f_{s})$ at $\rho_s$.
As the Milnor number of $E(P,f_s)$ at a singular point $\varrho_s\in\mbox{Sing}(E(P,f_s))\subseteq\hat E(P)^*$ coincides with the Milnor number of $V_P(f_s)$ at $\psi(\varrho_s)\in\mathbb{P}^{2*}(P)$ (where $\psi\colon \hat E(P)^*\to\mathbb{P}^{2*}(P)$ is the analytic isomorphism \eqref{analyiso}), the relation \eqref{sumle} implies that 
\begin{equation*}
\mu^{\text{\tiny tot}}(E(P,f_s))<\mu^{\text{\tiny tot}}(E(P,f_{s_0})), 
\end{equation*}
which contradicts Lemma \ref{ctmn}. Thus the family $\{f_s\}$ indeed does not have any bifurcation of singularities. Therefore, $V_P(f_s)$ has exactly $k$ singular points $\rho_{s,1},\ldots,\rho_{s,k}$ in $\mathbb{P}^{2*}(P)$, and the corresponding $k$ Milnor numbers $\mu(V_P(f_{s}),\rho_{s,i})$, $1\leq i\leq k$, are independent of $s$.
Thus, by a theorem of L\^e \cite{Le2}, for any $s$ the germs 
\begin{equation*}
(V_P(f_s),\rho_{s,i})
\quad\mbox{and}\quad
(V_P(f_0),\rho_{0,i})
\end{equation*}
 are topologically equivalent.
As $f_0$ belongs to $\mathcal{W}_{P,d}'(\Delta,\Xi)$, so does $f_s$ for all $s$. In other words, $g_s\in\mathcal{W}_{P,d,m}''(\Delta,\Xi)$ for all $s$.

\section{Proof of Proposition \ref{prop-samecc}}\label{proof-prop-samecc}

 Let $\psi\colon \mathbb{P}^2\to\mathbb{P}^2(P)$ be the $p_1p_2p_3$-fold branched covering introduced in \eqref{mapphi}.
Consider the group $G:=\mu_{p_1}\times\mu_{p_2}\times\mu_{p_3}$, where $\mu_{p_i}$ is the group of $p_i$th roots of unity. Clearly, $G$ acts on $\mathbb{P}^2$ by 
\begin{equation*}
(\xi_1,\xi_2,\xi_3)\circ [x\!:\!y\!:\!z]:=[\xi_1x\!:\!\xi_2y\!:\!\xi_3z],
\end{equation*}
and $\psi$ is the quotient map by this group action, that is, the inverse image of a point in $\mathbb{P}^2(P)$ by $\psi$ is the orbit of a point in $\mathbb{P}^2$ under this action.
Note that this action preserves the toric stratifications of $\mathbb{P}^2$ and $V_{P_0}(F_j) \subseteq \mathbb{P}^2$---obtained by replacing $P$ by $P_0:={}^t(1,1,1)$ and $f_j$ by the pull-back $F_j:=\psi^*f_j\colon (x,y,z)\mapsto f_j(x^{p_1},y^{p_2},z^{p_3})$ in \eqref{toricstratifications}---in the sense that each stratum  is mapped to itself.

Put $\mathcal{I}:=[0,1]$, and pick a continuous path 
\begin{equation*}
\gamma\colon s\in\mathcal{I}\mapsto \gamma(s):=f_s(x,y,z)\in\mathcal{W}'_{P,d}(\Delta,\Xi)
\end{equation*}
which connects $f_0$ and $f_1$.
Since the $V_P(f_s)$'s have the same singularities $\xi_1,\ldots\xi_k\in\Xi$ as the parameter $s$ varies, it follows that for each $1\leq i\leq k$, the family  consisting  of all function-germs $(f_s,\rho_{s,i})$ for $s\in \mathcal{I}$ and $\rho_{s,i}\in\mbox{Sing}(V_{P}(f_s))\cap\mathbb{P}^{2*}(P)$ is a $\mu^*$-constant family.
Moreover, since $\mathcal{W}'_{P,d}(\Delta,\Xi)$ is constructible (see Proposition \ref{prop-constructible}), we may assume that the above path $\gamma$ is \emph{piecewise complex-analytic} (see \cite[Remark~5.3]{EO2}). Refer to footnote~\ref{f14} for the definition, noting that in the present context $\mathcal{W}''_{P,d,m}(\Delta)$ is to be replaced by $\mathcal{W}'_{P,d}(\Delta,\Xi)$. 
Hereafter, for simplicity, we shall assume that the integer $q_0$ appearing in this footnote is equal to 1, and that the path $\gamma$ is a \emph{complex-analytic} path defined on an open subset $U_0\subseteq\mathbb{C}$ containing $\mathcal{I}$. (Although slightly more technical, the general case follows the same pattern.) 
Now, by \cite[Theorem 5.4 and Proposition~4.2]{EO2},  there exists a $\mu^*$-constant complex-analytic family of function-germs parametrized by $s\in U_0$, which coincides with the original family $\{(f_s,\rho_{s,i})\}_{s\in\mathcal{I}}$ when the parameter $s$ varies within $\mathcal{I}$. 
Similarly, there is a complex-analytic path 
\begin{equation}\label{mscfF}
s\in U_0\mapsto F_s(x,y,z):=f_s(x^{p_1},y^{p_2},z^{p_3})
\end{equation}
which defines a $\mu^*$-constant complex-analytic family of function-germs parametrized by $s\in U_0$, and which coincides, for $s\in \mathcal{I}$, with the family of function-germs $(F_s,\check\rho_{s,i})$, where $\check\rho_{s,i}$ is one of the $p_1 p_2 p_3$ singular points of $V_{P_0}(F_s)\subseteq\mathbb{P}^2$ that is mapped to $\rho_{s,i}$ by $\psi$.

We aim to construct an \emph{admissible homeomorphism} between the pairs $(\mathbb{P}^{2}(P),V_P(f_0))$ and $(\mathbb{P}^{2}(P),V_P(f_1))$. By a partition of unity argument, it is enough to show that for each $0\leq a\leq 1$, there exist $\varepsilon>0$ and a continuous family of admissible homeomorphisms
\begin{align}\label{familyHomeo}
\phi_s\colon (\mathbb{P}^{2}(P),V_P(f_{a}))\to (\mathbb{P}^{2}(P),V_P(f_s)) 
\end{align}
for $s\in \mathcal{I}_{a,\varepsilon}^{\pm}$, where $\mathcal{I}_{a,\varepsilon}^{\pm}$ denotes either the interval $[a,a+\varepsilon]$ or the interval $[a-\varepsilon,a]$.
A classical approach to achieving this typically involves constructing a continuous vector field on $\mathcal{I}_{a,\varepsilon}^{\pm}\times \mathbb{P}^2(P)$ that is compatible with the subset
\begin{equation*}
 V_{a,\varepsilon}^{\pm}(f):=\{(s,[x\!:\!y\!:\!z]_P)\in\mathcal{I}_{a,\varepsilon}^{\pm}\times \mathbb{P}^2(P)\, ;\, f(s,x,y,z):=f_s(x,y,z)=0\}
\end{equation*} 
as well as with its restriction to the coordinate subspaces, and subsequently integrating the field.
However, due to the presence of singularities in $\mathbb{P}^2(P)$, such a direct construction proves to be challenging. To overcome this difficulty, we first construct a vector field on $\mathcal{I}_{a,\varepsilon}^{\pm} \times \mathbb{P}^2$ that possesses analogous properties and is, moreover,  $G$-invariant. The flow generated by this vector field then descends to a well-defined continuous flow on the quotient space $\mathcal{I}_{a,\varepsilon}^{\pm}\times \mathbb{P}^2(P)$, whose integration yields the desired continuous family of admissible homeomorphisms \eqref{familyHomeo}.
The possibility to carrying out this procedure is ensured by the following lemma in which we first construct such a vector field \emph{locally}.

\begin{lemma}\label{lemws}
Let $(a,\mathbf{p})\in\mathcal{I}\times\mathbb{P}^2$. There exist $\varepsilon>0$, a neighbourhood $N_{\mathbf{p}}$ of $\mathbf{p}$ in $\mathbb{P}^2\equiv \mathbb{P}^2(P_0)$, and a continuous stratified vector field $\mathcal{V}_{a,\mathbf{p}}$ on $\mathcal{I}_{a,\varepsilon}^{\pm}\times N_{\mathbf{p}}$ such that:
\begin{enumerate}
\item
$\mathcal{V}_{a,\mathbf{p}}$ respects the natural stratification of $\mathcal{I}_{a,\varepsilon}^{\pm}\times N_{\mathbf{p}}$ induced by 
\begin{equation*}
 V_{a,\varepsilon}^{\pm}(F):=\{(s,[x\!:\!y\!:\!z])\in\mathcal{I}_{a,\varepsilon}^{\pm}\times \mathbb{P}^2\, ;\, F(s,x,y,z):=F_s(x,y,z)=0\}
\end{equation*} 
and its singular set, as well as its restrictions to the coordinate subspaces 
\begin{equation*}
\mathcal{I}_{a,\varepsilon}^{\pm}\times (N_{\mathbf{p}}\cap \mathbb{P}^I(P_0)),
\end{equation*} 
where $I\subseteq\{1,2,3\}$, $I\not=\emptyset$;
\item
the pushforward of $\mathcal{V}_{a,\mathbf{p}}$ by the first projection $\omega\colon \mathcal{I}_{a,\varepsilon}^{\pm}\times N_{\mathbf{p}}\to \mathcal{I}_{a,\varepsilon}^{\pm}$ is given~by 
\begin{equation*}
\omega_*\mathcal{V}_{a,\mathbf{p}}=\frac{\partial}{\partial s}.
\end{equation*} 
\end{enumerate}
\end{lemma}

\begin{proof}
It suffices to consider the case where $\mathbf{p}$ is a singular point $\check\rho_{a,i}$ ($1\leq i\leq k$). At any other point, the construction is clear.
Take local analytic coordinates $(u,v)$ near $\mathbf{p}\equiv\check\rho_{a,i}$. 
By Puiseux's theorem for space curves (see \cite{Maurer}), there exists an integer $m>0$ such that $\check\rho_{s,i}=(u(s), v(s))$ admits a local expansion of the form
\begin{equation}\label{trajectory}
\check\rho_{s,i}=\sum_{n \geq 0} c_n (s-a)^{n/m}, \ c_n\in\mathbb{C}^2,
\end{equation}
which converges in a sufficiently small neighbourhood $\mathcal{I}_{a,\varepsilon}:=[a-\varepsilon,a+\varepsilon]$ of $a$.
Let us denote by $C_{a,\varepsilon,i}:=\{(s,\check\rho_{s,i})\, ;\, s\in \mathcal{I}_{a,\varepsilon}\}$ the trajectory of the singular point $\check\rho_{s,i}$ defined by \eqref{trajectory}.
Then, consider the new parameter $t:=(s-a)^{1/m}$, so that this trajectory expands in a convergent power series in $t$. 
Note that for even values of $m$, this parametrization accounts only for the right-hand side of $C_{a,\varepsilon,i}$, namely for $C_{a,\varepsilon,i}^+:=\{(s,\check\rho_{s,i})\, ;\, s\in \mathcal{I}_{a,\varepsilon}^+:=[a,a+\varepsilon]\}$. Thus, the two sides---$C_{a,\varepsilon,i}^+$ on one hand, and the left-hand side $C_{a,\varepsilon,i}^-:=\{(s,\check\rho_{s,i})\, ;\, s\in \mathcal{I}_{a,\varepsilon}^-:=[a-\varepsilon,a]\}$ on the other hand---must be considered separately. Since the argument proceeds identically in both cases, we will restrict our attention to $C_{a,\varepsilon,i}^+$.

In the original coordinates $(s,u,v)$, the trajectory $C_{a,\varepsilon,i}^+$ is given by
\begin{equation*}
C_{a,\varepsilon,i}^+=\{(s,u(s), v(s))\, ;\, s\in \mathcal{I}_{a,\varepsilon}^+\}.
\end{equation*}
We now apply a change of coordinates that straightens $C_{a,\varepsilon,i}^+$. More precisely, consider the coordinates $(t,u',v')$ defined by $u':=u-u(a+t^m)$, $v':=v-v(a+t^m)$. Then, clearly, in these new coordinates, the trajectory is given by
\begin{equation*}
C_{a,\varepsilon,i}^+=\{(t,0,0)\, ;\, 0\leq t\leq\varepsilon^{1/m}\}.
\end{equation*}
Since the complex-analytic family \eqref{mscfF}---and consequently the corresponding family $F':=\{F_{a+t^m}\}$ in the new coordinates $(u',v')$---is $\mu^*$-constant, it follows from a theorem of Teissier \cite[Chap.~II, Th\'eor\`eme 3.9]{Teissier2} that there exist neighbourhoods $U$ and $N$ of $0\in \mathbb{C}$ and $\mathbf{0}\in \mathbb{C}^2$, respectively, such that $U\times N$ admits a Whitney $(b)$-regular stratification which is compatible with the family $F'$ and its singular locus. Then, by the Thom-Mather theorem, it follows that $F'$ is topologically trivial along its singular set. In particular, the family of function-germs $F_{a+t^m}$ for $t\in [0,\varepsilon^{1/m}]$ is topologically trivial along $[0,\varepsilon^{1/m}]\times \{\mathbf{0}\}$. Coming back to  the original parameter $s$ and coordinates $(u,v)$, we deduce that there exists a neighbourhood $N_{\mathbf{p}}$ of $\mathbf{p}\equiv\check\rho_{a,i}\in \mathbb{P}^2$ such that the family consisting of the $F_s$'s for $s\in \mathcal{I}^{+}_{a,\varepsilon}$ is topologically trivial along $C^{+}_{a,\varepsilon,i}\cap(\mathcal{I}^{+}_{a,\varepsilon}\times N_\mathbf{p})$.
This trivial structure enables the construction of a continuous stratified vector field on $\mathcal{I}^{+}_{a,\varepsilon}\times N_\mathbf{p}$ that satisfies conditions (1) and (2) of the lemma.
\end{proof}

Since no distinct singular points $\check\rho_{s,i}\not=\check\rho_{s',i'}$ merge into a single one (i.e., there is no bifurcation of singularities, as in the proof of Theorem \ref{mt2}) and since the above construction is trivial at any other point $\mathbf{p}\in\mathbb{P}^2$, it follows from Lemma \ref{lemws}, together with a standard partition of unity argument, that there exists a \emph{global} continuous stratified vector field $\mathcal{V}$ on $\mathcal{I}^{\pm}_{a,\varepsilon}\times\mathbb{P}^2$ such that:
\begin{enumerate}
\item
$\mathcal{V}$ respects the natural stratification of $\mathcal{I}^{\pm}_{a,\varepsilon}\times\mathbb{P}^2$ induced by $V_{a,\varepsilon}^{\pm}(F)$ and its singular set, as well as its restriction to the coordinate subspaces $\mathcal{I}^{\pm}_{a,\varepsilon}\times \mathbb{P}^I(P_0)$;
\item
$\omega_*\mathcal{V}=\frac{\partial}{\partial s}$.
\end{enumerate}

Now, consider the \emph{averaged} vector field
\begin{equation*}
\mathcal{V}^G:=\frac{1}{p_1p_2p_3}\sum_{g\in G} g_*\mathcal{V}
\end{equation*}
associated with $\mathcal{V}$.
Since  the $G$-action respects the stratifications of $\mathcal{I}^{\pm}_{a,\varepsilon}\times \mathbb{P}^2$ and $\mathcal{I}^{\pm}_{a,\varepsilon}\times \mathbb{P}^I(P_0)$, the pushforward $g_*\mathcal{V}$ satisfies properties (1) and (2) for all $g\in G$. It follows that $\mathcal{V}^G$ inherits these two properties as well while also being $G$-invariant.  
As its flow then commutes with the $G$-action on $\mathcal{I}_{a,\varepsilon}^{\pm}\times \mathbb{P}^2$, it descends to a continuous flow on the quotient $\mathcal{I}_{a,\varepsilon}^{\pm}\times \mathbb{P}^2(P)$, thereby yielding the desired continuous family of admissible homeomorphisms described in \eqref{familyHomeo}. More precisely, let~$\check\phi_s$ and $\phi_s$, $s\in \mathcal{I}_{a,\varepsilon}^{\pm}$, denote the one-parameter families of diffeomorphisms of $\mathbb{P}^2$ and $\mathbb{P}^2(P)$, generated by $\mathcal{V}^G$ and $\psi_*\mathcal{V}^G$, respectively. These families satisfy the compatibility condition 
\begin{equation*}
\psi \circ \check\phi_s = \phi_s \circ \psi,
\end{equation*}
which uniquely characterizes $\phi_s$.

\section{Proof of Theorem \ref{mt3}}\label{proofmt3}

By Remark \ref{rem-cpca} (b) and Theorem \ref{remthm1}, the polynomials $g_0$ and $g_1$ have the same Teissier $\mu^*$-invariant. Moreover, by a weighted homogeneous version of \cite[Theorem 25 and Remark 26]{O2}, they have diffeomorphic links (see Remark \ref{rem-cpca} (c)). Additionally, we easily deduce from \cite{O3,O2}, through formula \eqref{okaformula} above, that they also have identical monodromy zeta-functions, as detailed below.
Take a regular simplicial cone subdivision $\Sigma^*$ of the dual Newton diagram $\Gamma^*(g_0)=\Gamma^*(g_1)$ such that $P$ is a vertex of $\Sigma^*$, and consider the corresponding toric modification $\pi\colon X\to \mathbb{C}^3$. Let $\sigma:=\mbox{Cone}(P,Q,R)$ be a cone of $\Sigma^*$ having $P$ as a vertex, and let $(\mathbb{C}^3_\sigma,(u,v,w))$ be the corresponding toric chart, so that $\hat E(P)$ is defined by $u=0$. As in the proof of Lemma~\ref{ctmn}, the functions $g_0$ and $g_1$ are weakly almost Newton non-degenerate, and the formula \eqref{okaformula} applies to both of them. The zeta-function $\zeta(t)$ that appears in the right-hand side of this formula, and which is completely determined by $P$ and $d$ by Milnor-Orlik's theorem \cite[Theorem 4]{MO}, is the same for both $g_0$ and $g_1$.  
Since $f_0$ and $f_1$ form a Zariski pair of curves in $\mathbb{P}^2(P)$, and therefore share  the same combinatoric, we also have $\mu^{\text{\tiny tot}}(E(P,f_0))=\mu^{\text{\tiny tot}}(E(P,f_1))$. Now, modulo re-ordering, for any $\varrho_j\in\mbox{Sing}(E(P,f_j))$, $j=0,1$, there exists $\xi_i\in\Xi$ such that the isomorphism class of the germ $(E(P,f_j),\varrho_{j})$ is $\xi_i$ for $j=0$ and for $j=1$ (see section~\ref{sect-mngps}). Since the singularities $\xi_i$ in $\Xi$ are defined by Newton non-degenerate functions $\varphi_i$, there exist admissible local coordinates $(u,v',w')$ near $\varrho_{j}$ in which the pull-back $\pi^*\! g_j$ is written as
\begin{equation*}
\pi^*\! g_j(u,v',w')=u^d \Big(\phi_j(v',w')+\sum_{\alpha\geq m}c_{\alpha\beta\gamma}\, u^\alpha {v'}^\beta{w'}^\gamma\Big),
\end{equation*}
where $m\in\mathbb{N}^*$, $c_{\alpha\beta\gamma}\in\mathbb{C}$ with $c_{m00}\not=0$, $\phi_j(v',w')$---the defining function of $E(P,f_j)$---is Newton non-degenerate at the origin, and $\Gamma(\phi_0)=\Gamma(\phi_1)=\Gamma(\varphi_i)$ in their respective local coordinates (see Definition \ref{def-spaceND}).
 By an argument similar to that given in \cite[Claim 4.1]{EO3}, it follows that $\pi^*\! g_0$ and $\pi^*\! g_1$ are Newton non-degenerate and with the same Newton boundary, and hence, by Varchenko's theorem \cite[Theorem (4.1)]{V}, the zeta-functions 
\begin{equation*}
\zeta_{\pi^*\! g_0,\,\varrho_{0}}(t)
\quad\mbox{and}\quad
\zeta_{\pi^*\! g_1,\,\varrho_{1}}(t)
\end{equation*}
involved in formula \eqref{okaformula} coincide. By this formula, we therefore have $\zeta_{g_0,\mathbf{0}}(t)=\zeta_{g_1,\mathbf{0}}(t)$.

To complete the proof, it remains to show that $g_0$ and $g_1$ cannot be joined by a $\mu$-constant continuous path in $\mathcal{W}''_{P,d,m}(\Delta)$. By Remark \ref{rem-cpca} (a), this is equivalent to showing that $g_0$ and $g_1$ cannot be connected by a piecewise complex-analytic path in the $\mu$-constant stratum of $\mathcal{W}''_{P,d,m}(\Delta)$.
We argue by contradiction. Suppose there exists a $\mu$-constant piecewise complex-analytic family $\{g_s\}_{0\leq s\leq 1}$ connecting $g_0$ and $g_1$ in $\mathcal{W}''_{P,d,m}(\Delta)$. As $g_0,g_1\in\mathcal{W}''_{P,d,m}(\Delta,\Xi)$, it follows from Theorem \ref{mt2} that $g_s\in\mathcal{W}''_{P,d,m}(\Delta,\Xi)$ for any $s$. In particular, the weighted homogeneous initial polynomial $f_s$ of $g_s$ is in $\mathcal{W}'_{P,d}(\Delta,\Xi)$. 
Thus, by Proposition \ref{prop-samecc}, there is an admissible homeomorphism between the pairs $(\mathbb{P}^2(P),V_P(f_0))$ and $(\mathbb{P}^2(P),V_P(f_1))$, so that $f_0$ and $f_1$ do not form a Zariski pair of curves in $\mathbb{P}^2(P)$---a contradiction.

\section*{Acknowledgments}

This research was supported by the Narodowe Centrum Nauki under the grant number
2023/49/B/ST1/00848.

\bibliographystyle{amsplain}

\begin{thebibliography}{10}

\bibitem{Arnold} V. I. Arnol'd, \textit{Normal forms of functions near degenerate critical points, the Weyl groups $A_k$, $D_k$, $E_k$ and Lagrangian singularities}, Funkcional. Anal. i Prilo\v{z}en. \textbf{6} (1972), no.~4, 3--25.

\bibitem{A2}  E. Artal Bartolo, \textit{Sur la monodromie des singularit\'es superisol\'ees}, C. R. Acad. Sci. Paris S\'er. I Math. \textbf{312} (1991), no. 8, 601--604.

\bibitem{Artal} E. Artal Bartolo, \textit{Sur les couples de Zariski}, J. Algebraic Geom. \textbf{3} (1994), 223--247.

\bibitem{A1}  E. Artal Bartolo, \textit{Forme de Jordan de la monodromie des singularit\'es superisol\'ees de surfaces}, Mem. Amer. Math. Soc. \textbf{109} (1994), no. 525.

\bibitem{AC} E. Artal Bartolo and J. Carmona Ruber, \textit{Zariski pairs, fundamental groups and Alexander polynomials}, J. Math. Soc. Japan \textbf{50} (1998), 521--543.

\bibitem{ACM} E. Artal Bartolo, J. I. Cogolludo-Agust\'{i}n and J. Mart\'{i}n-Morales, \textit{Cremona transformations of weighted projective planes, Zariski pairs, and rational cuspidal curves}, in: Singularities and their interaction with geometry and low dimensional topology -- in honor of Andr\'as N\'emethi, pp. 117--157, Trends Math., Birkh\"auser/Springer, Cham, 2021.

\bibitem{ACT} E. Artal Bartolo, J. Cogolludo, H. Tokunaga, \textit{A survey on Zariski pairs}, in: Algebraic geometry in East Asia--Hanoi 2005, pp.~1--100, Adv. Stud. Pure Math. \textbf{50},
Math. Soc. Japan, Tokyo, 2008.

\bibitem{ABLMH} E. Artal Bartolo, J. Fern\'andez de Bobadilla, I. Luengo and A. Melle-Hern\'andez, \textit{Milnor number of weighted-L\^e-Yomdin singularities}, Int. Math. Res. Not. IMRN 2010, no. 22, 4301--4318.

\bibitem{A} N. A'Campo, \textit{La fonction z\^eta d'une monodromie}, Comment. Math. Helv. \textbf{50} (1975), 233--248.

\bibitem {BO} S. Brzostowski and G. Oleksik, \textit{On combinatorial criteria for non-degenerate singularities}, Kodai Math. J. \textbf{39} (2016), no.~2, 455--468.

\bibitem{CC} C. Chevalley and H. Cartan, \textit{Sch\'emas normaux; morphismes; ensembles constructibles}, S\'eminaire Henri Cartan \textbf{8} (1955/56), G\'eom\'etrie Alg\'ebrique, no. 7, 1--10.

\bibitem {Dimca} A. Dimca, \textit{Singularities and topology of hypersurfaces}, Universitext, Springer, New York, 1992.

\bibitem{Deg} A. Degtyarev, \textit{Oka's conjecture on irreducible plane sextics}, J. London Math. Soc. (2) \textbf{78} (2008), no.~2, 329--351.

\bibitem{EO4} C. Eyral and M. Oka,  \textit{Alexander-equivalent Zariski pairs of irreducible sextics}, J. Topol. \textbf{2} (2009), no.~3, 423--441.

\bibitem{EO5} C. Eyral and M. Oka,  \textit{On the geometry of certain irreducible non-torus plane sextics}, Kodai Math. J. \textbf{32} (2009), no.~3, 404--419.

\bibitem {EO} C. Eyral and M. Oka, \textit{$\mu^*$-Zariski pairs of surface singularities}, Nagoya Math. J. \textbf{254} (2024), 488--497.

\bibitem {EO2} C. Eyral and M. Oka, \textit{On paths in the $\mu$-constant and $\mu^*$-constant strata}, Hiroshima Math. J. \textbf{54} (2024), no.~2, 157--167.

\bibitem {EO3} C. Eyral and M. Oka, \textit{Milnor number of certain weighted L\^e-Yomdin hypersurface singularities}, Kyushu J. Math. (to appear).

\bibitem{Grothendieck} A. Grothendieck, \textit{\'El\'ements de g\'eom\'etrie alg\'ebrique, IV, \'Etude locale des sch\'emas et des morphismes de sch\'emas, I}, Inst. Hautes \'Etudes Sci. Publ. Math. \textbf{20} (1964), 259 pp.

\bibitem{GZ-L-MH} S. M. Gusein-Zade, I. Luengo and A. Melle-Hern\'andez, \textit{Partial resolutions and the zeta-function of a singularity}, Comment. Math. Helv. \textbf{72} (1997), no. 2, 244--256.

\bibitem{H}  C. Ha\c s Bey, \textit{Sur l'irr\'eductibilit\'e de la monodromie locale; application \`a l'\'equisingularit\'e}, C. R. Acad. Sci. Paris S\'er. A--B \textbf{275} (1972), A105--A107. 

\bibitem{I} I. N. Iomdin, \textit{Complex surfaces with a one-dimensional set of singularities}, (in Russian) Sibirsk. Mat.~\v{Z}. \textbf{15} (1974), 1061--1082, 1181. English translation: Siberian Math. J. \textbf{15} (1974), no. 5, 748--762 (1975).

\bibitem{K} A. G. Kouchnirenko, \textit{Poly\`edres de Newton et nombres de Milnor},
Invent. Math. \textbf{32} (1976), no. 1, 1--31.

\bibitem{Lazzeri} F. Lazzeri, \textit{A theorem on the monodromy of isolated singularities},  in: Singularit\'es \`a Carg\`ese (Rencontre Singularit\'es G\'eom. Anal., Inst. \'Etudes Sci. de Carg\`ese, 1972), pp. 269--275, Ast\'erisque \textbf{7} \& \textbf{8}, Soc. Math. France, Paris, 1973.

\bibitem{Le2} L\^e D\~ung Tr\'ang, \textit{Sur un crit\`ere d'\'equisingularit\'e}, C. R. Acad. Sci. Paris S\'er. A-B \textbf{272} (1971), A138--A140.

\bibitem{Le} L\^e D\~ung Tr\'ang, \textit{Une application d'un th\'eor\`eme d'A'Campo \`a l'\'equisingularit\'e}, Nederl. Akad. Wetensch. Proc. Ser. A \textbf{76} = Indag. Math. \textbf{35} (1973), 403--409.

\bibitem{LeDT} L\^e D\~ung Tr\'ang, \textit{Ensembles analytiques complexes avec lieu singulier de dimension un (d'apr\`es I. N. Iomdine)}, Seminar on Singularities (Paris, 1976/1977), pp. 87--95, Publ. Math. Univ. Paris VII, \textbf{7}, Univ. Paris VII, Paris, 1980.

\bibitem{Luengo} I. Luengo, \textit{The $\mu$-constant stratum is not smooth}, Invent. Math. \textbf{90} (1987), no. 1, 139--152.

\bibitem{Maurer} J. Maurer, \textit{Puiseux expansion for space curves}, Manuscripta Math. \textbf{32} (1980), no. 1--2, 91--100.

\bibitem{MO} J. Milnor and P. Orlik, \textit{Isolated singularities defined by weighted homogeneous polynomials}, Topology~\textbf{9} (1970), 385--393.

\bibitem{O6} M. Oka, \textit{Symmetric plane curves with nodes and cusps}, J. Math. Soc. Japan \textbf{44} (1992), no.~3, 375--414.

\bibitem{O5} M. Oka, \textit{Two transforms of plane curves and their fundamental groups},
J. Math. Sci. Univ. Tokyo \textbf{3} (1996), no.~2, 399--443.

\bibitem{O1} M. Oka, \textit{Non-degenerate complete intersection singularity}, Actualit\'es Math., Hermann, Paris, 1997. 

\bibitem{O-1} M. Oka, \textit{Flex curves and their applications}, Geometriae Dedicata \textbf{75} (1999), 67--100.

\bibitem{O-2} M. Oka, \textit{A new Alexander-equivalent Zariski pair}, Acta Math. Vietnam. \textbf{27} (3) (2002), 349--357. 

\bibitem{O3} M. Oka, \textit{Almost non-degenerate functions and a Zariski pair of links}, in: Essays in geometry \textemdash\ dedicated to Norbert A'Campo, pp. 601--628, IRMA Lect. Math. Theor. Phys. \textbf{34}, EMS Press, Berlin, 2023.

\bibitem{O2} M. Oka, \textit{On $\mu$-Zariski pairs of links}, J. Math. Soc. Japan \textbf{75} (2023), no. 4, 1227--1259.

\bibitem{O4} M. Oka, \textit{A comment on Brian\c{c}on-Speder polynomial}, J. Math. Soc. Japan \textbf{77} (2025), no.~2, 499--511.

\bibitem{Siersma} D. Siersma, \textit{The monodromy of a series of hypersurface singularities}, Comment. Math. Helv. \textbf{65} (1990), no. 2, 181--197.

\bibitem{Siersma2} D. Siersma, \textit{The vanishing topology of non isolated singularities}, New developments in singularity theory (Cambridge, 2000), pp.~447--472, NATO Sci. Ser. II Math. Phys. Chem., \textbf{21}, Kluwer Acad. Publ., Dordrecht, 2001.

\bibitem{Stevens} J. Stevens, \textit{On the $\mu$-constant stratum and the $V$-filtration: an example}, Math. Z. \textbf{201} (1989), no. 1, 139--144.

\bibitem{Teissier2} B. Teissier, \textit{Cycles \'evanescents, sections planes et conditions de Whitney}, in: Singularit\'es \`a Carg\`ese (Rencontre Singularit\'es G\'eom. Anal., Inst. \'Etudes Sci., Carg\`ese, 1972), pp. 285--362, Ast\'erisque \textbf{7} \& \textbf{8}, Soc. Math. France, Paris, 1973.

\bibitem{V} A. N. Varchenko, \textit{Zeta-function of monodromy and Newton's diagram}, Invent. Math. \textbf{37} (1976), no.~3, 253--262.

\bibitem{Yau} S.~S.~T. Yau, Topological type of isolated hypersurface singularities, in: Recent developments in geometry (Los Angeles, CA, 1987), pp.~303--321, Contemp. Math. \textbf{101}, Amer. Math. Soc., Providence, RI, 1989.

\bibitem{Z} O. Zariski, \textit{On the problem of existence of algebraic functions of two variables possessing a given branched curve}, Amer. J. Math. \textbf{51} (1929), no. 2, 305--328.

\end{thebibliography}

\end{document}